\documentclass[11pt]{article}
\usepackage{amsmath,amsfonts,amsthm,amssymb, mathtools}
\usepackage{mathrsfs, graphicx,color,latexsym, tikz, calc,
}
\usetikzlibrary{shadows}
\usetikzlibrary{patterns,arrows,decorations.pathreplacing}
\newtheorem{theorem}{Theorem}
\newtheorem{lemma}{Lemma}

\newtheorem{corollary}{Corollary}
\newtheorem{remark}{Remark}

\title{\bf \Large Mixed graphs with smallest eigenvalue greater than $-\frac{\sqrt{5}+1}{2}$}

\author{
{\small  Lu Lu$^{a}$,\ \ Zhenzhen Lou$^{b}$\footnote{Corresponding author.
\newline{\it \hspace*{5mm}Email addresses:} xjdxlzz@163.com (Z. Lou), lulugdmath@163.com (L. Lu).},\ \ 
Qiongxiang Huang $^{b}$}\\[2mm]
\footnotesize $^a$ School of Mathematics and Statistics, Central South University, Changsha, Hunan, 410083, China\\
\footnotesize $^b$ College of Mathematics and Systems Science, Xinjiang University, Xinjiang, Urumqi, 830046, China}
\date{ }

\begin{document}

\maketitle

\begin{abstract}
The classical problem of characterizing the graphs with bounded eigenvalues may date back to the work of Smith in 1970. Especially, the research on graphs with smallest eigenvalues not less than $-2$ has attracted widespread attention. Mixed graphs are natural generalization of undirected graphs.
In this paper, we completely characterize the mixed graphs with smallest Hermitian eigenvalue greater than
$-\frac{\sqrt{5}+1}{2}$, which consists of three infinite classes of mixed graphs and $30$ scattered mixed graphs. By the way, we get a new class of mixed graphs switching equivalent to their underlying graphs.\\

\end{abstract}

\baselineskip=0.202in

\section{Introduction}\label{s-1}
It is a classical problem in Spectral Graph Theory to characterize the graphs whose eigenvalues are bounded. The research of
such problems may date back to the work of Smith in 1970 \cite{Smith}. This work stimulated the interest of researchers.
There are a lot of results in the literature concerning the topic. In 1972, Hoffman \cite{Hoffman1} obtained all limit
points of the spectral radius of non-negative symmetric matrices smaller than $\frac{\sqrt{5}+1}{2}$. In 1982, Cvetkovi\'{c}
et al. \cite{Cvetkovic1} characterized the graphs whose spectral radius does not exceed $\sqrt{2+\sqrt{5}}$ and in 1989,
Brouwer and Neumaier \cite{Brouwer1} determined the graphs with spectral radius between $2$ and $\sqrt{2+\sqrt{5}}$ and
later, Woo and Neumaier \cite{Woo} described the structure of graphs whose spectral radii are bounded above by
$3\sqrt{2}/2$. With respect to the smallest eigenvalues, Hoffman \cite{Hoffman2} investigated the graphs whose smallest
eigenvalue exceeds $-1-\sqrt{2}$, and this work was continued by Taniguchi et al. \cite{Tani1,Tani2,Tani3}. Furthermore, Munemasa et al. \cite{Munemasa} showed that all fat Hoffman graphs with smallest eigenvalue at least $-\frac{\sqrt{5}+1}{2}$ (which is just $-1-\tau$ where $\tau$ is the golden ratio) can be described by a finite set of fat $(-1-\tau)$-irreducible Hoffman graphs. Especially, the
graphs with smallest eigenvalue $-2$ attracted a lot of attention, and we refer the reader to the survey \cite{Cvetkovic2}
and the book \cite{Cvetkovic3}. Recently, Abdollahi et al. \cite{Abdollahi} classified all distance-regular Cayley graphs with least eigenvalue $-2$ and diameter at most three, and Koolen et al. \cite{Koolen} proved that a connected graph with smallest eigenvalue at least $-3$ and large enough minimal degree is $2$-integrable. In this paper we consider the smallest Hermitian eigenvalues of mixed graphs.

A mixed graph is defined to be an ordered triple $(V,E,A)$, where $V$ is the vertex set, $E$ is the undirected edge set and
$A$ is the directed edge set. Note that, if both $uv$ and $vu$ are directed edges, then we regard $\{u,v\}$ an undirected
edge.
Thus, if $(u,v)\in A$ then $(v,u)\not\in A$. Clearly, if $A=\emptyset$ then the mixed graph turns to be a graph and if
$E=\emptyset$ then the mixed graph turns to be an oriented graph. For convenience, we write $u\leftrightarrow v$ if
$\{u,v\}\in E$
and $u\rightarrow v$ is $(u,v)\in A$. Let $M=(V,E,A)$ be a mixed graph with $V=\{v_1,v_2,\ldots,v_n\}$. The {\it underlying
graph} $\Gamma(M)$ is a graph with vertex set $V$ and two vertices $u\sim v$ if $u\leftrightarrow v$ or $u\rightarrow v$ or
$v\rightarrow u$. For $U\subseteq V$ and $W\in V\setminus U$, denote by $N_W(U)=\{w\mid w\in W,u\sim w\textrm{ in $G$ for
some $u\in U$}\}$. Especially, if $U=\{u\}$ then $N_W(u)$ is the set of neighbors of $u$ in $W$. Moreover, denote by
$N_W^+(u)=\{w\mid u\rightarrow w\}$, $N_W^-(u)=\{w\mid u\leftarrow w\}$ and $N_W^o(u)=\{w\mid u\leftrightarrow w\}$. It is
clear that $N_W(u)=N_W^+(u)\cup N_W^-(u)\cup N_W^o(u)$. As usual, we always write $P_n$, $C_n$, $K_{n_1,n_2,\ldots,n_k}$ the
path, the cycle and the complete multipartite graph of the corresponding order. For two graphs $G$ and $H$, the {\it union
$G\cup H$} is the graph with vertex set $V(G)\cup V(H)$ and edge set $E(G)\cup E(H)$. The \emph{join} $G\nabla H$ is the
graph obtained from $G\cup H$ by adding all edges between $G$ and $H$. The {\it distane} of two vertices $u,v\in V(G)$ in
$G$ is the length of a shortest path from $u$ to $v$ in $G$, denoted by
$d_G(u,v)$. The {\it diameter} of $G$ is the largest distance in $G$, denoted by $d(G)$. All other notions not mentioned
here are standard in \cite{Godsil}.

We always write $M_G$ for $M$ when the underlying graph $\Gamma(M)=G$. Moreover, for a graph $G$, denote by $\mathcal{M}_G$
the set of mixed graphs with underlying graph $G$. Especially, if $M_G=G$ then we write $G$ for $M_G$. The mixed graph $M_G$
is {\it connected} if $G$ is connected and we
always consider the connected mixed graphs in this paper. The {\it diameter} of $M_G$ is defined to be the diameter of $G$,
denoted by $d(M_G)$. For a subset $U\subseteq V$, the mixed subgraph induced by $U$ is the mixed graph $M_G[U]=(U,E',A')$
with $E'=\{\{u,v\}\mid u,v\in U, \{u,v\}\in E\}$ and $A'=\{(u,v)\mid u,v\in U,(u,v)\in A\}$. As usual, for a vertex $v$,  the (mixed) graph $G-v$ (resp. $M_G-v$) is the induced (resp. mixed) subgraph obtained from $G$ (resp. $M_G$) by deleting the vertex $v$ and associated edges. The {\it Hermitian matrix} of
$M_G$ is defined to be a square
matrix $H(M_G)=[h_{st}]_{n\times n}$ with
\[h_{st}=\left\{\begin{array}{cccc}
1,& v_s\leftrightarrow v_t,\\
i,& v_s\rightarrow v_t,\\
-i,&v_t\rightarrow v_s,\\
0, &\textrm{ otherwise},
\end{array}\right.\]
which was proposed by Liu and Li \cite{Liu} and Guo and Mohar \cite{Guo} independently. Since $H(M_G)$ is a Hermitian
matrix, all eigenvalues of $H(M_G)$ are real and listed as $\lambda_1\ge\lambda_2\ge\cdots\ge\lambda_n$. The collection of
such eigenvalues is the spectrum of $H(M_G)$. The {\it Hermitian spectrum} of the mixed graph $M_G$ is just the spectrum of
$H(G)$, denoted by $\operatorname{Sp}(M_G)$. Two mixed graphs $M_G,M_G'\in \mathcal{M}_G$ are {\it switching equivalent} if
there exists a diagonal matrix $D$ whose entries belong to $\{\pm1,\pm i\}$ such that $H(M_G')=DH(M_G)D^*$. It is clear that
the relation switching equivalence is an equivalent relation. Thus, denote by $[M_G]$ the equivalence class containing
$M_G$ with respect to switching equivalence. Obviously, all graphs in $[M_G]$ share the same spectrum. Recently,  Wissing and Dam \cite{Wissing} determined all mixed graphs with exactly one negative eigenvalue. Guo and
Mohar \cite{Guo2} determined all mixed graphs with $\lambda_1<2$ and Yuan et al. \cite{Yuan} characterized all mixed graphs
with $\lambda_1\le 2$ when $G$ contains no cycles of length $4$. 

In this paper, we completely determine the connected mixed graphs with smallest Hermitian eigenvalue greater than
$-\frac{\sqrt{5}+1}{2}$, which consists of three infinite classes and $30$ scattered mixed graphs (see Theorem \ref{thm-f-4}). As a byproduct, we get an interesting type of mixed graphs switching equivalent to their underlying graphs (see Theorem \ref{lem-x-1}).

\section{Preliminaries}
In this part, we will introduce some results which will be used latter. We first present the famous interlacing theorem with
respect to Hermitian matrix.
\begin{lemma}[\cite{Bollobas}]\label{lem-f-1}
Let the matrix $S$ of size $m \times n$ be such that $S^*S = I_m$ and let $H$ be a Hermitian
matrix of size $n$ with eigenvalues $\lambda_1\geq\lambda_2\ge\cdots\ge\lambda_n$. Set $B =S^* HS$ and
let $\mu_1\geq\mu_2\ge\cdots\ge\mu_m$ be the eigenvalues of $B$.
Then the eigenvalues of $H$ and $B$ are interlaced, that is,
$\lambda_i\geq\mu_i\geq\lambda_{n-m+i}$ for $i=1,2,...,m$.
\end{lemma}

The following result is immediate from Lemma \ref{lem-f-1}.
\begin{corollary}\label{cor-f-1}
Let $M_G$ be a mixed graph with underlying graph $G$. If $M_H$ is a mixed induced subgraph of $M_G$, then the eigenvalues of
$M_H$ interlace those of $M_G$.
\end{corollary}

Next we introduce another powerful tool in spectral graph theorem, that is the equitable partition. Let $M_G$ be a mixed
graph on $n$ vertices with underlying graph $G$. Let $\pi$: $V(G)=V_1\cup V_2\cup \cdots\cup V_s$ be a partition of $V(G)$
with $|V_i|=n_i$ and $n=n_1+n_2+\cdots+n_s$. For $1\le i,j\le s$, denote by $H_{i,j}$
the submatrix of $H(M_G)$ whose rows corresponding to $V_i$ and columns corresponding to $V_j$. Therefore, the Hermitian
matrix $H(M_G)$ can be written as
$H(M_G)=[H_{ij}]$. Denote by $b_{ij}=e^TH_{ij}e/n_i$ the average row-sums of $A_{ij}$, where $e$ denotes the all-one vector.
The matrix
$H_{\pi}=(b_{ij})_{s\times s}$ is called the {\it quotient matrix} of $H(M_G)$. If, for any $i,j$, the row-sum of $H_{ij}$
corresponding to any vertex $v\in V_i$
equals to $b_{ij}$, then $\pi$ is called an {\it equitable partition} of $M_G$. Let $\delta_{V_i}$ be a vector indexed by
$V(G)$ such that $\delta_{V_i}(v)=1$ if
$v\in V_i$ and $0$ otherwise. The matrix $P=[\delta_{V_1} \delta_{V_2}\cdots \delta_{V_s}]$ is called the {\it
characteristic matrix} of $\pi$. If $\pi$ is an
equitable partition, then $H(M_G)P=PH_{\pi}$. It leads to the following famous result.

\begin{lemma}[{\cite[Theorem 9.3.3, page 197]{Godsil}}]\label{lem-f-2}
Let $M_G$ be a mixed graph and $\pi$ an equitable partition of $M_G$ with quotient matrix $H_{\pi}$ and characteristic
matrix $P$. Then the eigenvalues of $H_{\pi}$ are also eigenvalues of $H(M_G)$. Furthermore, $H(M_G)$ has the following two
kinds of eigenvectors:
\begin{itemize}
\item[{\rm (i)}]
the eigenvectors in the column space of $P$, and the corresponding eigenvalues coincide with the eigenvalues of $H_{\pi}$;
\item[{\rm (ii)}]
the eigenvectors orthogonal to the columns of $P$, i.e., those eigenvectors sum to zero on each cell of $\pi$.
\end{itemize}
\end{lemma}

Let $\mathcal{H}$ be a set of graphs. A graph $G$ is called $\mathcal{H}$-free if non of graphs in $\mathcal{H}$ can be an
induced subgraph of $G$. Especially, if $\mathcal{H}=\{H\}$ then the $\mathcal{H}$-free graph $G$ is also called an $H$-free
graph. Recall that a $P_4$-free graph is called a cograph. The following result reveals the structure of cographs.
\begin{lemma}[\cite{Seinsche}]\label{lem-f-3}
If $G$ is a connected $P_4$-free graph, then $G$ is the join of two graphs, that is,  $G=G_1\nabla G_2$ for some graphs
$G_1$ and $G_2$ with $|V(G_1)|,|V(G_2)|\ge 1$.
\end{lemma}

We determine two types of $\mathcal{H}$-free graphs when $\mathcal{H}$ contains some simple graphs.
\begin{lemma}\label{lem-f-4}
If $G$ is a $\{  {K_{1,2}}, {3K_1}, {K_2\cup K_1}\}$-free graph then ${G}\in\{2K_1,K_n\mid n\ge 1\}$; if $G$ is a
$\{{K_{1,2}},{3K_1},  {K_3}\}$-free graph then ${G}\in\{K_1,K_2,2K_1,2K_2,K_1\cup K_2\}$.
\end{lemma}
\begin{proof}
It is clear that, if a graph $G$ is ${K_{1,2},3K_1}$-free, then it is the union of at most two complete graphs. Thus, we
have $G\in\{2K_1,K_n\mid n\ge 1\}$ if $G$ is additional $K_2\cup K_1$-free, and $G\in\{K_1,K_2,2K_1,2K_2,K_2\cup K_1,K_2\cup
K_2\}$
if $G$ is additional $K_3$-free.
\end{proof}

Guo and Mohar introduced the so called four-way switching to generate switching equivalent graphs \cite{Guo}. A {\it
four-way switching} is the operation of changing a mixed graph $M_G$ into the mixed graph $M'_G$ by choosing an appropriate
diagonal matrix $S$ with $S_{jj}\in\{\pm1,\pm i\}$ and setting $H(M'_G)=S^{-1}H(M_G)S$. Let $G$ be a graph and $X$ an edge
cut such that $G-X=G_1\cup G_2$ and $V_1=V(G_1)$ and $V_2=V(G_2)$. For a mixed graph $M_G=(V,E,A)$, define $X^+=\{(v_1,v_2)\mid\{v_1,v_2\}\in X,v_1\in
V_1,v_2\in V_2\}$ and $X^-=\{(v_2,v_1)\mid\{v_1,v_2\}\in X,v_1\in V_1,v_2\in V_2\}$. The cut $X$ is called a {\it coincident
cut} of the mixed graph $M_G$ if $X^+\subseteq A$ or $X^-\subseteq A$ or $X\subseteq E$, that is, the directions of the edges between $V_1$ and $V_2$ are coincident. If $X$ is a coincident cut of
$M_G$,
then the {\it $X$-switching} of $M_G$ is the mixed graph $M_G[X]=(V,E',A')$ with $E'=E\cup X$ and $A'=A\setminus(X^+\cup
X^-)$.
Note that $M_G[X]=M_G$ if $X\subseteq E$. From four-way switching, the following results are obtained.
\begin{lemma}[\cite{Guo}]\label{lem-f-5}
Let $M_G$ be a mixed graph. If $X$ is a coincident cut of $M_G$, then $M_G$ and $M_G[X]$ are switching equivalent and thus
$\operatorname{Sp}(M_G)=\operatorname{Sp}(M_G[X])$.
\end{lemma}
If $G$ is a forest, then each edge is a cut. Moreover, each edge is a coincident cut of any mixed graph $M_G$. Thus, Lemma
\ref{lem-f-5} implies the following result.
\begin{corollary}[\cite{Guo}]\label{cor-f-2}
If $G$ is a forest, then $\operatorname{Sp}(M_G)=\operatorname{Sp}(G)$ for any mixed graph $M_G\in\mathcal{M}_G$.
\end{corollary}

\section{Mixed graphs with $\lambda_n>-\frac{\sqrt{5}+1}{2}$}
In this part, we first investigate the mixed triangles in mixed graphs with underlying graph being complete. Next, we get
all mixed graphs with smallest eigenvalue not less than $-\sqrt{2}$. At last, we completely determine the mixed graphs with
smallest eigenvalue greater than $-\frac{\sqrt{5}+1}{2}\approx -1.618$.

\begin{figure}[htbp]
\begin{center}
\unitlength 3.2mm 
\linethickness{0.4pt}
\ifx\plotpoint\undefined\newsavebox{\plotpoint}\fi 
\begin{picture}(32,13)(0,0)
\thicklines
\put(8,12){\line(-3,-4){3}}
\put(5,8){\line(1,0){6}}
\put(11,8){\line(-3,4){3}}
\put(13,8){\line(3,4){3}}
\put(16,12){\line(3,-4){3}}
\put(16,8){\vector(1,0){.07}}\put(13,8){\line(1,0){6}}
\put(22.5,10){\vector(3,4){.07}}\multiput(21,8)(.03370787,.04494382){89}{\line(0,1){.04494382}}
\put(25.5,10){\vector(3,-4){.07}}\multiput(24,12)(.03370787,-.04494382){89}{\line(0,-1){.04494382}}
\put(21,8){\line(1,0){6}}
\put(2.5,3){\vector(3,4){.07}}\multiput(1,1)(.03370787,.04494382){89}{\line(0,1){.04494382}}

\put(5.5,3){\vector(-3,4){.07}}\multiput(7,1)(-.03370787,.04494382){89}{\line(0,1){.04494382}}
\put(1,1){\line(1,0){6}}
\put(10.5,3){\vector(-3,-4){.07}}\multiput(12,5)(-.03370787,-.04494382){89}{\line(0,-1){.04494382}}
\put(13.5,3){\vector(3,-4){.07}}\multiput(12,5)(.03370787,-.04494382){89}{\line(0,-1){.04494382}}
\put(9,1){\line(1,0){6}}
\put(18.5,3){\vector(3,4){.07}}\multiput(17,1)(.03370787,.04494382){89}{\line(0,1){.04494382}}
\put(21.5,3){\vector(3,-4){.07}}\multiput(20,5)(.03370787,-.04494382){89}{\line(0,-1){.04494382}}
\put(20,1){\vector(-1,0){.07}}\put(23,1){\line(-1,0){6}}
\put(26.5,3){\vector(3,4){.07}}\multiput(25,1)(.03370787,.04494382){89}{\line(0,1){.04494382}}
\put(29.5,3){\vector(3,-4){.07}}\multiput(28,5)(.03370787,-.04494382){89}{\line(0,-1){.04494382}}
\put(28,1){\vector(1,0){.07}}\put(25,1){\line(1,0){6}}
\put(8,7){\makebox(0,0)[cc]{\scriptsize$K_3,-1$}}
\put(16,7){\makebox(0,0)[cc]{\scriptsize$K_3^1,-\sqrt{3}$}}
\put(24,7){\makebox(0,0)[cc]{\scriptsize$K_3^{2,1},-2$}}
\put(4,0){\makebox(0,0)[cc]{\scriptsize$K_3^{2,2},-1$}}
\put(12,0){\makebox(0,0)[cc]{\scriptsize$K_3^{2,3},-1$}}
\put(20,0){\makebox(0,0)[cc]{\scriptsize$K_3^{3,1},-\sqrt{3}$}}
\put(28,0){\makebox(0,0)[cc]{\scriptsize$K_3^{3,2},-\sqrt{3}$}}
\put(8,12){\circle*{.5}}
\put(5,8){\circle*{.5}}
\put(11,8){\circle*{.5}}
\put(16,12){\circle*{.5}}
\put(13,8){\circle*{.5}}
\put(19,8){\circle*{.5}}
\put(24,12){\circle*{.5}}
\put(21,8){\circle*{.5}}
\put(27,8){\circle*{.5}}
\put(4,5){\circle*{.5}}
\put(1,1){\circle*{.5}}
\put(7,1){\circle*{.5}}
\put(9,1){\circle*{.5}}
\put(12,5){\circle*{.5}}
\put(15,1){\circle*{.5}}
\put(20,5){\circle*{.5}}
\put(17,1){\circle*{.5}}
\put(23,1){\circle*{.5}}
\put(25,1){\circle*{.5}}
\put(28,5){\circle*{.5}}
\put(31,1){\circle*{.5}}
\end{picture}
\begin{picture}(48,14)(0,0)
\thicklines
\put(3,12){\vector(1,0){.07}}\put(1,12){\line(1,0){4}}
\put(5,10){\vector(0,-1){.07}}\put(5,12){\line(0,-1){4}}
\put(1,12){\line(0,-1){4}}
\put(1,8){\line(1,0){4}}
\put(5,8){\circle*{0.5}}
\put(1,8){\circle*{0.5}}
\put(1,12){\circle*{0.5}}
\put(5,12){\circle*{0.5}}
\put(10,12){\vector(1,0){.07}}\put(8,12){\line(1,0){4}}
\put(8,12){\line(0,-1){4}}
\put(8,8){\line(1,0){4}}
\put(12,8){\circle*{0.5}}
\put(8,8){\circle*{0.5}}
\put(8,12){\circle*{0.5}}
\put(12,12){\circle*{0.5}}
\put(17,12){\vector(1,0){.07}}\put(15,12){\line(1,0){4}}
\put(19,10){\vector(0,-1){.07}}\put(19,12){\line(0,-1){4}}
\put(15,12){\line(0,-1){4}}
\put(15,8){\line(1,0){4}}
\put(19,8){\circle*{0.5}}
\put(15,8){\circle*{0.5}}
\put(15,12){\circle*{0.5}}
\put(19,12){\circle*{0.5}}
\put(22,12){\line(0,-1){4}}
\put(22,8){\line(1,0){4}}
\put(26,8){\circle*{0.5}}
\put(22,8){\circle*{0.5}}
\put(22,12){\circle*{0.5}}
\put(26,12){\circle*{0.5}}
\put(31,12){\vector(1,0){.07}}\put(29,12){\line(1,0){4}}
\put(33,10){\vector(0,-1){.07}}\put(33,12){\line(0,-1){4}}
\put(29,12){\line(0,-1){4}}
\put(33,8){\circle*{0.5}}
\put(29,8){\circle*{0.5}}
\put(29,12){\circle*{0.5}}
\put(33,12){\circle*{0.5}}
\put(3,5){\vector(1,0){.07}}\put(1,5){\line(1,0){4}}
\put(1,5){\line(0,-1){4}}
\put(1,1){\line(1,0){4}}
\put(5,1){\circle*{0.5}}
\put(1,1){\circle*{0.5}}
\put(1,5){\circle*{0.5}}
\put(5,5){\circle*{0.5}}
\put(12,3){\vector(0,-1){.07}}\put(12,5){\line(0,-1){4}}
\put(8,5){\line(0,-1){4}}
\put(8,1){\line(1,0){4}}
\put(12,1){\circle*{0.5}}
\put(8,1){\circle*{0.5}}
\put(8,5){\circle*{0.5}}
\put(12,5){\circle*{0.5}}
\put(17,5){\vector(1,0){.07}}\put(15,5){\line(1,0){4}}
\put(15,5){\line(0,-1){4}}
\put(15,1){\line(1,0){4}}
\put(19,1){\circle*{0.5}}
\put(15,1){\circle*{0.5}}
\put(15,5){\circle*{0.5}}
\put(19,5){\circle*{0.5}}
\put(24,5){\vector(1,0){.07}}\put(22,5){\line(1,0){4}}
\put(22,5){\line(0,-1){4}}
\put(26,1){\circle*{0.5}}
\put(22,1){\circle*{0.5}}
\put(22,5){\circle*{0.5}}
\put(26,5){\circle*{0.5}}
\put(29,5){\line(0,-1){4}}
\put(29,1){\line(1,0){4}}
\put(33,1){\circle*{0.5}}
\put(29,1){\circle*{0.5}}
\put(29,5){\circle*{0.5}}
\put(33,5){\circle*{0.5}}
\put(38,12){\vector(1,0){.07}}\put(36,12){\line(1,0){4}}
\put(40,10){\vector(0,-1){.07}}\put(40,12){\line(0,-1){4}}
\put(36,12){\line(0,-1){4}}
\put(36,8){\line(1,0){4}}
\put(40,8){\circle*{0.5}}
\put(36,8){\circle*{0.5}}
\put(36,12){\circle*{0.5}}
\put(40,12){\circle*{0.5}}
\put(45,12){\vector(1,0){.07}}\put(43,12){\line(1,0){4}}
\put(43,12){\line(0,-1){4}}
\put(43,8){\line(1,0){4}}
\put(47,8){\circle*{0.5}}
\put(43,8){\circle*{0.5}}
\put(43,12){\circle*{0.5}}
\put(47,12){\circle*{0.5}}
\put(36,5){\line(0,-1){4}}
\put(36,1){\line(1,0){4}}
\put(40,1){\circle*{0.5}}
\put(36,1){\circle*{0.5}}
\put(36,5){\circle*{0.5}}
\put(40,5){\circle*{0.5}}
\put(47,3){\vector(0,-1){.07}}\put(47,5){\line(0,-1){4}}
\put(43,5){\line(0,-1){4}}
\put(43,1){\line(1,0){4}}
\put(47,1){\circle*{0.5}}
\put(43,1){\circle*{0.5}}
\put(43,5){\circle*{0.5}}
\put(47,5){\circle*{0.5}}
\put(12,11){\line(0,1){0}}
\put(12,12){\line(0,-1){4}}
\put(10,8){\vector(-1,0){.07}}\put(12,8){\line(-1,0){4}}
\put(16.5,8){\vector(-1,0){.07}}\put(18,8){\line(-1,0){3}}
\put(26,12){\line(0,-1){4}}
\put(29,9.5){\vector(0,-1){.07}}\put(29,11){\line(0,-1){3}}
\put(31,8){\vector(1,0){.07}}\put(29,8){\line(1,0){4}}
\put(37.5,8){\vector(-1,0){.07}}\put(39,8){\line(-1,0){3}}
\put(36,9.5){\vector(0,1){.07}}\put(36,8){\line(0,1){3}}
\put(47,10){\vector(0,1){.07}}\put(47,8){\line(0,1){4}}
\put(5,3){\vector(0,1){.07}}\put(5,1){\line(0,1){4}}
\put(1,2.5){\vector(0,-1){.07}}\put(1,4){\line(0,-1){3}}
\put(3,1){\vector(-1,0){.07}}\put(5,1){\line(-1,0){4}}
\put(10,5){\vector(-1,0){.07}}\put(12,5){\line(-1,0){4}}
\put(15,1){\line(1,0){2}}
\put(19,5){\line(0,-1){4}}
\put(17,1){\vector(1,0){.07}}\put(15,1){\line(1,0){4}}
\put(44,7){\makebox(0,0)[cc]{\scriptsize$-2$}}
\put(31,0){\makebox(0,0)[cc]{\scriptsize$-1.8477$}}
\put(26,3){\vector(0,-1){.07}}\put(26,5){\line(0,-1){4}}
\put(24,1){\vector(-1,0){.07}}\put(26,1){\line(-1,0){4}}
\put(33,5){\line(0,-1){4}}
\put(31,5){\vector(1,0){.07}}\put(29,5){\line(1,0){4}}
\put(38,5){\vector(1,0){.07}}\put(36,5){\line(1,0){4}}
\put(38,1){\vector(1,0){.07}}\put(36,1){\line(1,0){4}}
\put(43,3){\vector(0,1){.07}}\put(43,1){\line(0,1){4}}
\put(45,5){\vector(-1,0){.07}}\put(47,5){\line(-1,0){4}}
\put(40,5){\line(0,-1){4}}
\put(22,12){\line(1,0){3}}
\put(17,7){\makebox(0,0)[cc]{\scriptsize$C_4^3, -\sqrt{2}$}}
\put(10,7){\makebox(0,0)[cc]{\scriptsize$C_4^2, -\sqrt{2}$}}
\put(3,7){\makebox(0,0)[cc]{\scriptsize$C_4^1, -\sqrt{2}$}}
\put(31,7){\makebox(0,0)[cc]{\scriptsize$-2$}}
\put(17,0){\makebox(0,0)[cc]{\scriptsize$-2$}}
\put(10,0){\makebox(0,0)[cc]{\scriptsize$-2$}}
\put(3,0){\makebox(0,0)[cc]{\scriptsize$-2$}}
\put(38,7){\makebox(0,0)[cc]{\scriptsize$-2$}}
\put(38,0){\makebox(0,0)[cc]{\scriptsize$-1.8477$}}
\put(45,0){\makebox(0,0)[cc]{\scriptsize$-1.8477$}}
\put(24,0){\makebox(0,0)[cc]{\scriptsize$-1.8477$}}
\put(24,7){\makebox(0,0)[cc]{\scriptsize$-2$}}
\put(22,12){\line(1,0){4}}
\end{picture}
\end{center}
\caption{\footnotesize{The mixed triangles and their smallest eigenvalues. }}\label{fig-1}
\end{figure}
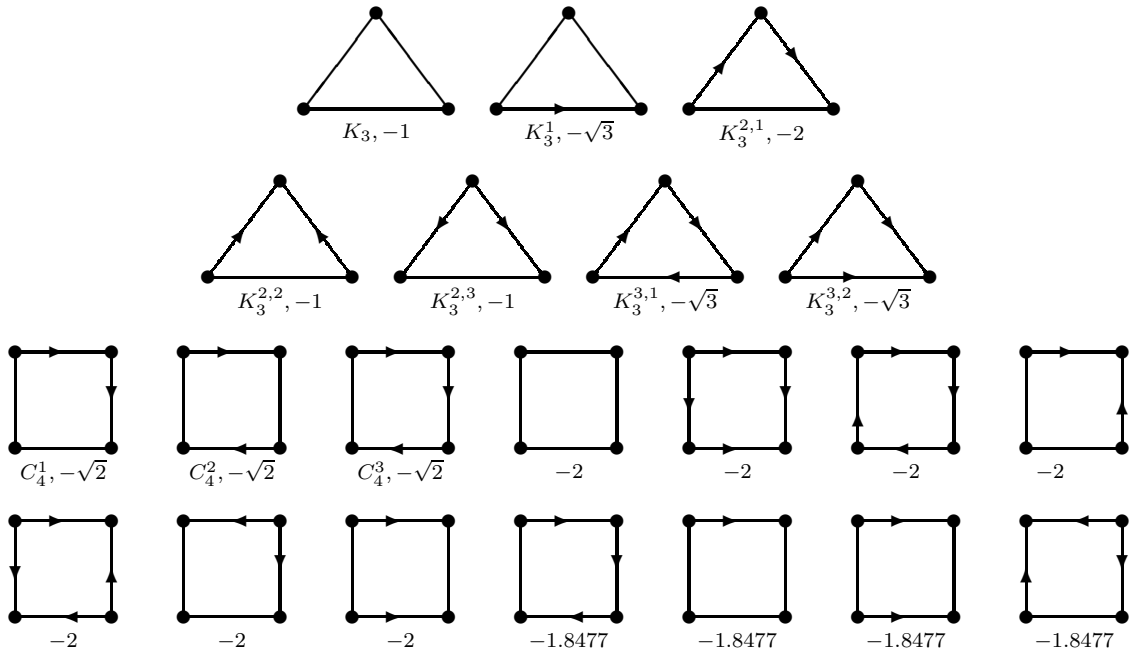

It is easy to verify that there are seven types of mixed triangles and fourteen types of mixed quadrangles, and we present
them in Fig.\ref{fig-1} together with their smallest eigenvalues. The following results are immediate from Lemma
\ref{lem-f-1} and Fig.\ref{fig-1}.
\begin{lemma}\label{lem-f-6}
Let $M_G$ be a mixed graph with smallest eigenvalue $\lambda_n$. If $\lambda_n>-\sqrt{3}$, then
 any mixed triangle in $M_G$ belongs to $\{K_3, K_3^{2,2}, K_3^{2,3}\}$.
\end{lemma}

\begin{lemma}\label{lem-x-2}
Let $M_G$ be a mixed graph with smallest eigenvalue $\lambda_n$. If $\lambda_n\ge-1.84$, then
 any induced mixed quadrangle in $M_G$ belongs to $\{ C_4^1,C_4^2,C_4^3\}$.
\end{lemma}

In what follows, we always denote $\mathcal{C}_3=\{K_3,K_3^{2,2}, K_3^{2,3}\}$ and $\mathcal{C}_4=\{ C_4^1,C_4^2,C_4^3\}$.
The mixed triangles $K_3$, $K_3^{2,2}$ and $K_3^{2,3}$ play an important role in determining the orientations of a mixed
graph, especially when all induced cycles (if exist) of the underlying graph are triangles.
Recall that a {\it chordal graph} is one in which all cycles of four or more vertices have a chord, which is an edge that is not part of the cycle but connects two vertices of the cycle. The following result characterizes a class of mixed chordal graphs switching equivalent to their underlying graphs.

\begin{theorem}\label{lem-x-1}
Let $G$ be a chordal graph. If $M_G$ is a mixed graph in which each mixed triangle belongs to $\mathcal{C}_3$, then $M_G\in[G]$, i.e., $M_G$ is switching equivalent to $G$.
\end{theorem}
\begin{proof}
Without loss of generality, assume that $G$ is connected. According to Lemma \ref{lem-f-5}, it suffices to show that $M_G$ has a coincident cut $X$ such that $V(G-X)=U\cup W$ satisfying that all edges in $U$ and $W$ are undirected. 

We prove the statement by induction on $n=|V(G)|$. The statement holds for $n=3$ clearly. Assume that the statement holds for $n-1$ with $n\ge 4$ and we prove it holds for $n$. It is well-known that a chordal graph has a perfect elimination ordering, which is an ordering of the vertices such that, for each vertex $v$, the vertex $v$ and the neighbors of $v$ that occur after $v$ in the order form a clique. Assume that $\{v_1,v_2,\ldots,v_n\}$ is a perfect ordering. By the inductive hypothesis, $M_G-v_1$ has a coincident cut $X$, say $V(G-v_1)=U\cup W$ such that the edges in $U$ and $W$ are undirected. According to the definition of $X$, we have either all edges between $U$ and $W$ are undirected or they have the same direction. Therefore, we divide two cases to discuss.

{\flushleft\bf Case 1.} All edge between $U$ and $W$ are undirected.

For any $u\in N_U(v_1)$ and $w\in N_W(v_1)$, since $v_1,u,w$ form a clique and $M_G[v_1,u_w]\in\mathcal{C}_3$, we have either $v_1\leftrightarrow u,w$, $v_1\rightarrow u,w$ or $u,w\rightarrow v_1$. If the first case occurs, then there is nothing to prove. If the second one occurs, then, for any $u'\in N_U(v_1)$, we have $v\rightarrow u'$ since $v_1,u,u'$ form a clique and $M_G[v_1,u,u']\in\mathcal{C}_3$. Similarly, we have $v_1\rightarrow w'$ for any $w'\in N_W(v_1)$. Therefore, all edges between $v_1$ and $U\cup W$ form the desired coincident cut. If the last one  occurs, one can similarly verify that all edges between $v_1$ and $U\cup W$ form the desired coincident cut.

{\flushleft\bf Case 2.} All edges between $U$ and $W$ have the same direction, say $u\rightarrow v$ for any $u\in U$ and $w\in W$ with $u\sim v$ in $G$.

For any $u\in N_U(v_1)$ and $w\in N_W(v_1)$, since $v_1,u,w$ form a clique and $M_G[v_1,u,w]\in\mathcal{C}_3$, we have either $v_1\leftrightarrow u$ and $v_1\rightarrow w$, or $u\rightarrow v_1$ and $v_1\rightarrow w$. If the former occurs, then, for any $u'\in N_U(v_1)$, we have $v_1\leftrightarrow u'$ since $v_1,u,u'$ form a clique and $M_G[v_1,u,u']\in\mathcal{C}_3$. Similarly, we have $v_1\rightarrow w'$ for any $w'\in N_W(v_1)$. Therefore, all edges between $\{v_1\}\cup U$ and $W$ form the desired coincident cut. If the latter occurs, one can similarly verify that all edges between $U$ and $\{v_1\}\cup W$ form the desired coincident cut.

The proof is completed.
\end{proof}

\begin{remark}
It is clear that Corollary \ref{cor-f-2} can be regarded as a special case of Theorem \ref{lem-x-1}. The characterization of mixed graphs switching equivalent to their underlying graphs is meaningful in itself. Mohar \cite{Mohar} investigated such problem recently. It is not hard to find another proof of Theorem \ref{lem-x-1} without using the perfect ordering by just analyzing the structure of $M_G$. It is fair to guess that there would be a more general result. 
\end{remark}

For non-negative integers $s,t,n$ with $n=s+t$, denote by $K_n[s,t]$ the mixed graph obtained from $K_s\cup K_t$ by adding
all arcs from the vertices of $K_s$ to those of $K_t$.
It is clear that $K_n[s,t]$ is switching equivalent to $K_n$. In fact, we will show that $[K_n]=\{K[s,t]\mid s,t\ge
0,s+t=n\}$ and give a characterization of graph set $[K_n]$.

\begin{lemma}\label{lem-f-7} Let $M_{K_n}$ be a mixed graph on $n\ge3$ vertices in which any mixed triangle belongs to
$\mathcal{C}_3$. If $M_{K_n}$ contains $K_3^{2,2}$, then $M_{K_n}\in\{K_n[s,t]\mid s\ge2,t\ge
1,s+t=n\}$.
\end{lemma}
\begin{proof}
Assume that $u,v, w\in V(M_{K_n})$ induce a $K_3^{2,2}$ with $u\rightarrow w$,$v\rightarrow w$ and $u\leftrightarrow v$. For
any vertex $x\in V(M_{K_n})\setminus \{u,v,
w\}$ (if exists),
we have either $x\leftrightarrow w$ or $x\rightarrow w$ since otherwise $M_{K_n}[u,w,x]\not\in\mathcal{C}_3$. By noticing
$M_{K_n}[u,x,w],M_{K_n}[v,x,w]\in\mathcal{C}_3$, one can easily verify that $v\rightarrow x$
and $u\rightarrow x$ if $x\leftrightarrow w$, and $x \leftrightarrow v$ and $u \leftrightarrow x$ if $x\rightarrow w$.

Denote by $V_1=\{x\in V(M_{K_n})\mid w\leftrightarrow x\}\cup\{w\}$ and $V_2=\{x\in V(M_{K_n})\mid w \leftarrow x\}$. It is
clear that we have $u,v\in V_2$ and $V=V_1\cup V_2$. For any two vertices $x_1,x_1'\in V_1\setminus\{w\}$, we have
$x_1\leftrightarrow x_1'$ since $x_1,x_1'\leftrightarrow w$ and $M_G[x_1,x_1',w]\in\mathcal{C}_3$. Similarly, we have
$x_2\leftrightarrow x_2'$ for any $x_2,x_2'\in V_2$. Moreover, for any $x_1\in V_1\setminus \{w\}$ and $x_2\in V_2$, we have
$x_2\rightarrow x_1$ since $x_1\leftrightarrow w$, $x_2\rightarrow w$ and $M_G[x_1,x_2,w]\in \mathcal{C}_3$.
Thus, $M_{K_n}= K_n[s,t]$ where $s=|V_2|\ge2$ and $t=|V_1|\ge 1$.
\end{proof}
Similarly, we get the following result.

\begin{lemma}\label{lem-f-8}
Let $M_{K_n}$ be a mixed graph on $n\ge3$ vertices in which any mixed triangular belongs to
$\mathcal{C}_3$. If $M_{K_n}$ contains $K_3^{2,3}$, then $M_{K_n}\in\{K_n[s,t]\mid s\ge1,t\ge
2,s+t=n\}$.
\end{lemma}
\begin{proof}
Assume that $u,v, \omega$ induce a $K_3^{2,3}$ with $u\leftarrow \omega$, $v\leftarrow \omega$ and $u\leftrightarrow v$.
For any vertex $x\in V(M_{K_n})\backslash \{u,v,\omega\}$( if exists),
we have either $x\leftrightarrow \omega$ or $x\leftarrow \omega$ since otherwise $M_{K_n}[u,w,x]\not\in\mathcal{C}_3$.
Note that $M_{K_n}[u,x,w],M_{K_n}[v,x,w]\in\mathcal{C}_3$.
We have $x\rightarrow u$ and $x\rightarrow v$ if $\omega \leftrightarrow x$, and
 $x\leftrightarrow u$ and $ x\leftrightarrow v$ if $\omega\rightarrow x$.
Let $V_3=\{x\in V(M_{K_n})\mid \omega \leftrightarrow x \}\cup
\{\omega\}$ and $V_4=\{x\in V(M_{K_n})\mid \omega\rightarrow x \}$.

Clearly, $V(M_{K_n})=V_3\cup V_4$, $u,v\in V_3$.
Taking $x_3,x_3'\in V_3$ and $x_4,x_4'\in V_4$, we get $x_3\leftrightarrow x_3'$ and $x_4\leftrightarrow x_4'$,
Therefore, $V_3$ and $V_4$ induce an clique, respectively, and $|V_3|\geq1$, $|V_4|\geq 2$.
Moreover, we also have $x_3\rightarrow x_4$ for any $x_3\in V_3, x_4\in V_4$.
Therefore, we get $M_{K_n}= K_n[s,t]$ with $s=|V_3|\geq 1$ and $t=|V_4|\geq2$.
\end{proof}

Lemmas \ref{lem-f-7} and \ref{lem-f-8} yield the following result.
\begin{theorem}\label{thm-f-1}
Let $M_{K_n}$ be a mixed graph with underlying graph $K_n$ and $n\ge 3$. Then any mixed triangle of $M_{K_n}$ belongs to
$\mathcal{C}_3$ if and only if $M_{K_n}\in\{K_{n}[s,t]\mid s,t\ge 0, s+t=n\}$ if and only if
$M_{K_n}\in[K_n]$.
\end{theorem}
\begin{proof}
Firstly, assume that any triangle  of $M_{K_n}$
belongs to $\mathcal{C}_3$.
Lemma \ref{lem-f-7} and Lemma \ref{lem-f-8} indicate that $M_{K_n}\in\{K_n[s,t]\mid s\ge 1,t\ge1, s+t=n\}$ when $M_G$
contains $K_3^{2,2}$ or $K_3^{2,3}$. If $M_{K_n}$ contains neither $K_3^{2,2}$ nor $K_3^{2,3}$, then any mixed triangle of
$M_{K_n}$ is $K_3$ and thus $M_{K_n}= K_n=K_{n}[n,0]$.  Conversely, one can easily verify that  any mixed triangle of
$K_n[s,t]$ belongs to $\mathcal{C}_3$.

Next we will show $[K_n]=\{K_n[s,t] \mid s,t\ge0, s+t=n\}$.
It is clear that $\{K_n[s,t]\mid s,t\ge0,s+t=n\}\subseteq [K_n]$. It suffices to show that $[K_n]\subseteq\{K_n[s,t]\mid
s,t\ge0,s+t=n\}$. By the arguments above, it only needs to show that any mixed triangle in $M_{K_n}$ belongs to
$\mathcal{C}_3$ for any $M_{K_n}\in[K_n]$. Assume that  $H(M_{K_n})=[h_{j,k}]$ for a mixed graph $M_{K_n}\in[K_n]$. Since
$M_{K_n}\in[K_n]$, there exists a diagonal matrix $D=diag(d_1,d_2,\ldots,d_n)$ with $d_j\in\{\pm i,\pm 1\}$ such that
$DH(M_{K_n})D^*=H(K_n)$. Therefore, for any
$\{u,v,w\}\subseteq V(M_{K_n})$, we have\[
\begin{pmatrix}d_u&&\\&d_v&\\&&d_w\end{pmatrix}
\begin{pmatrix}0&h_{uv}&\overline{h}_{wu}\\ \overline{h}_{uv}&0&h_{vw}\\h_{wu}&\overline{h}_{vw}&0\end{pmatrix}
\begin{pmatrix}\overline{d_u}&&\\&\overline{d_v}&\\&&\overline{d_w}\end{pmatrix}=
\begin{pmatrix}0&1&1\\1&0&1\\1&1&0\end{pmatrix}.
\]
It leads to $d_uh_{uv}\overline{d_v}=1$, $d_vh_{vw}\overline{d_w}=1$ and $d_wh_{wu}\overline{d_u}=1$. Thus, we have
$h_{uv}h_{vw}h_{wu}=1$. It implies that either exactly one of $h_{uv},h_{vw},h_{wu}$ equal to $1$ or all of them equal to
$1$. If the former case happens, say $h_{uv}=1$, then $\{h_{vw},h_{wu}\}=\{\pm i\}$, which means $M_{K_n}[u,v,w]=
K_{3}^{2,2}$ or $K_3^{2,3}$. If the latter case happens, then $M_{K_n}[u,v,w]= K_{3}$.

The proof is completed.
\end{proof}

Now we give a simple application of Theorem \ref{thm-f-1} as follows.
\begin{theorem}\label{thm-f-2}
Let $M_G$ be a connected mixed graph on $n$ vertices.
Then $\lambda_n(M_G)> -\sqrt{2}$
if and only if $M_{G}\in\{K_{n}[s,t]\mid s,t\ge 0, s+t=n\}$.
\end{theorem}
\begin{proof}
Theorem \ref{thm-f-1} implies that the mixed graph $K_n[s,t]$ has the spectrum $\{n-1,[-1]^{n-1}\}$, and the sufficiency
follows. Now we consider the necessity.
Assume that $M_G$ is a mixed graph on $n$ vertices with $\lambda_n(M_G)> -\sqrt{2}$. Since
$\operatorname{Sp}(M_{P_3})=\operatorname{Sp}(P_3)=\{\pm \sqrt{2},0\}$, the path $P_3$ cannot be an induced subgraph of $G$
due to Corollary \ref{cor-f-1}. Thus, we have $G= K_n$. Furthermore, since $\lambda_n(M_G)>-\sqrt{2}>-\sqrt{3}$,  Lemma
\ref{lem-f-6} also implies that each
triangle in $M_G$ belongs to $\mathcal{C}_3$. Thus,  we have $M_{G}\in\{K_{n}[s,t]\mid s,t\ge 0, s+t=n\}$
by Theorem \ref{thm-f-1}.
\end{proof}

Theorem \ref{thm-f-2} gives the characterization of mixed graphs with $\lambda_n>-\sqrt{2}$. In what follows, we will
further determine the mixed graphs with $\lambda_n\geq -\sqrt{2}$.
\begin{lemma}\label{lem-f-9}
Let $M_G$ be a connected mixed graph on $n$ vertices. If $\lambda_n(M_G)\geq -\sqrt{2}$, then $G$ is $\{P_3\nabla
K_1,(K_2\cup
K_1)\nabla K_1\}$-free.
\end{lemma}
\begin{proof}
Suppose to the contrary that $G$ contains induced $H$ for $H\in \{P_3\nabla K_1,(K_2\cup
K_1)\nabla K_1\}$. Therefore, Corollary \ref{cor-f-1} means that $\lambda_4(M_H)\ge-\sqrt{2}>-\sqrt{3}$, and thus each mixed
triangle of $M_H$ belongs to $\mathcal{C}_3$. Note that $H$ has no cycle with length greater than $3$. Theorem \ref{lem-x-1}
implies that $\lambda_4(M_H)=\lambda_4(H)$, which equals to $\lambda_4(P_3\nabla K_1)=-1.56<-\sqrt{2}$ or
$\lambda_4((K_2\cup K_1)\nabla K_1)=-1.48<-\sqrt{2}$, a contradiction.
\end{proof}

By Lemma \ref{lem-f-9}, we get the following result.
\begin{theorem}\label{thm-f-3}
Let $M_G$ be a connected mixed graph on $n\ge 4$ vertices.
Then $\lambda_n(M_G)\geq -\sqrt{2}$ if and only if $M_{K_n}\in\{K_{n}[s,t]\mid s,t\ge 0, s+t=n\}\cup\mathcal{C}_4$.
\end{theorem}
\begin{proof}
The sufficiency is immediate and we show the necessity in what follows. We divide two cases to discuss.

{\bf Case 1.} $G$ is $P_3$-free.

In this case, we have $G= K_n$. Since $\lambda_n(M_G)\geq -\sqrt{2}>-\sqrt{3}$, any mixed triangle in $M_G$ belongs to
$\mathcal{C}_3$ by Lemma \ref{lem-f-6}. Thus, Theorem \ref{thm-f-1} means $M_{G}\in\{K_{n}[s,t]\mid s,t\ge 0,
s+t=n\}$.

{\bf Case 2.} $G$ is not $P_3$-free.

In this case, suppose that there exists $u,v,w\in V(G)$ such that $G[u,v,w]={P_3}$ with $u\sim v$ and $v\sim w$. Note that
$\lambda_4(P_4)\approx-1.618<-\sqrt{2}$ and $\lambda_4(K_{1,3})=-\sqrt{3}<-\sqrt{2}$. Corollary \ref{cor-f-1} implies that
$G$ is $\{P_4,K_{1,3}\}$-free, and thus the diameter $d(G)=2$. Therefore, each vertex $y\in V(G)\setminus\{u,v,w\}$ of
$V(G)$ is adjacent to at least one vertex of $\{u,v,w\}$. If $y$ is adjacent to exactly one vertex of $\{u,v,w\}$, then $G$
either contains an induced $P_4$ or $K_{1,3}$, which is impossible. If $y$ is adjacent to all the vertices $\{u,v,w\}$, then
$G[u,v,w,y]=P_3\bigtriangledown K_1$, which contradicts Lemma \ref{lem-f-9}. Thus, $y$ is adjacent to exactly two vertices
of $\{u,v,w\}$. If $y\sim u,v$ or $y\sim v,w$, then $G[u,v,w,y]=(K_2\cup K_1)\nabla K_1$, which contradicts Lemma
\ref{lem-f-9}. Thus, $y\sim u, w$, that is $G[u,v,w,y]=C_4$. Next, we claim that $n=4$. Otherwise, there exists another
vertex $y'\in V(G)\setminus\{u,v,w,y\}$. By regarding $y'$ as $y$, we have $G[u,v,w,y']=C_4$. Therefore, we have
$G[u,v,y,y']=K_{1,3}$ when $y\nsim y'$ and $G[u,v,y,y']=(K_2\cup K_1)\bigtriangledown K_1$ when $y\sim y'$, which
are all impossible. Therefore, we have $G=C_4$, and thus $M_G\in\mathcal{C}_4$ by Fig.\ref{fig-1}.

This completes the proof.
\end{proof}

In what follows, we characterize the mixed graph $M_G$ with $\lambda_n(M_G)> -\frac{1+\sqrt{5}}{2}$.
We first determine the underlying graph of $G$.

\begin{lemma}\label{lem-f-10}
If $M_{G}$ be a mixed graph with underling graph $G=K_{m,n}$,
then $\lambda_n(M_{G})\leq -\frac{1+\sqrt{5}}{2}$ except for $G= K_{2},K_{1,2}$ or $K_{2,2}$.
\end{lemma}
\begin{proof}
Assume $\lambda_n(M_{G})> -\frac{1+\sqrt{5}}{2}$, then
$G$ has no $K_{1,3}$ as an induced subgraph since $\lambda_3(K_{1,3})= -\sqrt{3}<-\frac{1+\sqrt{5}}{2}\approx-1.618$.
This leads to $G= K_{2},K_{1,2}$ or $K_{2,2}$. It follows the result.
\end{proof}

\begin{figure}[htbp]
\begin{center}
\unitlength 3.5mm 
\linethickness{0.4pt}
\ifx\plotpoint\undefined\newsavebox{\plotpoint}\fi 
\begin{picture}(35,28)(0,0)
\thicklines
\put(1,27){\line(1,0){4}}
\put(5,27){\line(1,0){2}}
\put(7,27){\line(0,-1){6}}
\put(7,21){\line(-1,0){6}}
\put(1,21){\line(0,1){6}}
\multiput(1,27)(.03370787,-.03370787){89}{\line(0,-1){.03370787}}
\multiput(4,24)(.03370787,.03370787){89}{\line(0,1){.03370787}}
\multiput(4,24)(-.03370787,-.03370787){89}{\line(0,-1){.03370787}}
\multiput(4,24)(.03370787,-.03370787){89}{\line(0,-1){.03370787}}
\put(1,27.5){\makebox(0,0)[cc]{\tiny$u_1$}}
\put(7,27.5){\makebox(0,0)[cc]{\tiny$u_2$}}
\put(1,20.5){\makebox(0,0)[cc]{\tiny$u_4$}}
\put(4,25){\makebox(0,0)[cc]{\tiny$v$}}
\put(13,27){\vector(1,0){.07}}\put(10,27){\line(1,0){6}}
\put(16,24){\vector(0,-1){.07}}\put(16,27){\line(0,-1){6}}
\put(13,21){\vector(-1,0){.07}}\put(16,21){\line(-1,0){6}}
\put(10,24){\vector(0,-1){.07}}\put(10,27){\line(0,-1){6}}
\multiput(9.93,26.93)(.6,-.6){11}{{\rule{.8pt}{.8pt}}}
\multiput(9.93,20.93)(.6,.6){11}{{\rule{.8pt}{.8pt}}}
\put(10,27.5){\makebox(0,0)[cc]{\tiny$u_1$}}
\put(16,25.5){\makebox(0,0)[cc]{\tiny$u_2$}}
\put(16,20.5){\makebox(0,0)[cc]{\tiny$u_3$}}
\put(10,20.5){\makebox(0,0)[cc]{\tiny$u_4$}}
\put(22,27){\vector(1,0){.07}}\put(19,27){\line(1,0){6}}
\put(25,24){\vector(0,-1){.07}}\put(25,27){\line(0,-1){6}}

\put(19,27){\line(0,-1){6}}
\put(19,21){\line(1,0){6}}
\multiput(18.93,26.93)(.6,-.6){11}{{\rule{.8pt}{.8pt}}}
\multiput(18.93,20.93)(.6,.6){11}{{\rule{.8pt}{.8pt}}}
\put(19,27.5){\makebox(0,0)[cc]{\tiny$u_1$}}
\put(25,27.5){\makebox(0,0)[cc]{\tiny$u_2$}}
\put(25,20.5){\makebox(0,0)[cc]{\tiny$u_3$}}
\put(19,20.5){\makebox(0,0)[cc]{\tiny$u_4$}}
\put(13,25){\makebox(0,0)[cc]{\tiny$v$}}
\put(22,25){\makebox(0,0)[cc]{\tiny$v$}}
\put(31,27){\vector(1,0){.07}}\put(28,27){\line(1,0){6}}
\put(31,21){\vector(-1,0){.07}}\put(34,21){\line(-1,0){6}}
\put(28,27){\line(0,-1){6}}
\put(34,27){\line(0,-1){6}}
\multiput(27.93,26.93)(.6,-.6){11}{{\rule{.8pt}{.8pt}}}
\multiput(33.93,20.93)(0,0){3}{{\rule{.8pt}{.8pt}}}
\multiput(33.93,26.93)(-.6,-.6){11}{{\rule{.8pt}{.8pt}}}
\put(28,27.5){\makebox(0,0)[cc]{\tiny$u_1$}}
\put(34,27.5){\makebox(0,0)[cc]{\tiny$u_2$}}
\put(34,20.5){\makebox(0,0)[cc]{\tiny$u_3$}}
\put(28,20.5){\makebox(0,0)[cc]{\tiny$u_4$}}
\put(31,25){\makebox(0,0)[cc]{\tiny$v$}}
\put(1,27){\circle*{.5}}
\put(7,27){\circle*{.5}}
\put(4,24){\circle*{.5}}
\put(1,21){\circle*{.5}}
\put(7,21){\circle*{.5}}
\put(10,27){\circle*{.5}}
\put(16,27){\circle*{.5}}
\put(13,24){\circle*{.5}}
\put(10,21){\circle*{.5}}
\put(16,21){\circle*{.5}}
\put(19,27){\circle*{.5}}
\put(19,21){\circle*{.5}}
\put(25,21){\circle*{.5}}
\put(28,27){\circle*{.5}}
\put(34,27){\circle*{.5}}
\put(28,21){\circle*{.5}}
\put(34,21){\circle*{.5}}
\put(22,24){\circle*{.5}}
\put(31,24){\circle*{.5}}
\put(7,20){\makebox(0,0)[cc]{\tiny$u_3$}}
\put(5,17){\vector(1,0){.07}}\put(2,17){\line(1,0){6}}
\put(8,14){\vector(0,-1){.07}}\put(8,17){\line(0,-1){6}}
\put(5,11){\vector(-1,0){.07}}\put(8,11){\line(-1,0){6}}
\put(2,14){\vector(0,-1){.07}}\put(2,17){\line(0,-1){6}}
\multiput(1.93,16.93)(.6,-.6){11}{{\rule{.8pt}{.8pt}}}
\multiput(1.93,10.93)(.6,.6){11}{{\rule{.8pt}{.8pt}}}
\put(2,17.5){\makebox(0,0)[cc]{\tiny$u_1$}}
\put(8,17.5){\makebox(0,0)[cc]{\tiny$u_2$}}
\put(8,10.5){\makebox(0,0)[cc]{\tiny$u_3$}}
\put(2,10.5){\makebox(0,0)[cc]{\tiny$u_4$}}
\put(5,15){\makebox(0,0)[cc]{\tiny$v$}}
\put(2,17){\circle*{.5}}
\put(8,17){\circle*{.5}}
\put(5,14){\circle*{.5}}
\put(2,11){\circle*{.5}}
\put(8,11){\circle*{.5}}
\put(17,17){\vector(1,0){.07}}\put(14,17){\line(1,0){6}}
\put(20,14){\vector(0,-1){.07}}\put(20,17){\line(0,-1){6}}
\put(17,11){\vector(-1,0){.07}}\put(20,11){\line(-1,0){6}}
\put(14,14){\vector(0,-1){.07}}\put(14,17){\line(0,-1){6}}
\multiput(13.93,16.93)(.6,-.6){11}{{\rule{.8pt}{.8pt}}}
\multiput(13.93,10.93)(.6,.6){11}{{\rule{.8pt}{.8pt}}}
\put(14,17.5){\makebox(0,0)[cc]{\tiny$u_1$}}
\put(20,17.5){\makebox(0,0)[cc]{\tiny$u_2$}}
\put(20,10.5){\makebox(0,0)[cc]{\tiny$u_3$}}
\put(14,10.5){\makebox(0,0)[cc]{\tiny$u_4$}}
\put(17,15){\makebox(0,0)[cc]{\tiny$v$}}
\put(14,17){\circle*{.5}}
\put(20,17){\circle*{.5}}
\put(17,14){\circle*{.5}}
\put(14,11){\circle*{.5}}
\put(20,11){\circle*{.5}}
\put(15.5,15.5){\vector(1,-1){.07}}\multiput(14,17)(.03370787,-.03370787){89}{\line(0,-1){.03370787}}
\multiput(20,17)(-.03370787,-.03370787){89}{\line(0,-1){.03370787}}
\put(9,14){\vector(1,0){4}}

\put(29,17){\vector(1,0){.07}}\put(26,17){\line(1,0){6}}
\put(32,14){\vector(0,-1){.07}}\put(32,17){\line(0,-1){6}}
\put(29,11){\vector(-1,0){.07}}\put(32,11){\line(-1,0){6}}
\put(26,14){\vector(0,-1){.07}}\put(26,17){\line(0,-1){6}}
\multiput(25.93,16.93)(.6,-.6){11}{{\rule{.8pt}{.8pt}}}
\multiput(25.93,10.93)(.6,.6){11}{{\rule{.8pt}{.8pt}}}
\put(26,17.5){\makebox(0,0)[cc]{\tiny$u_1$}}
\put(32,17.5){\makebox(0,0)[cc]{\tiny$u_2$}}
\put(32,10.5){\makebox(0,0)[cc]{\tiny$u_3$}}
\put(26,10.5){\makebox(0,0)[cc]{\tiny$u_4$}}
\put(29,15){\makebox(0,0)[cc]{\tiny$v$}}
\put(26,17){\circle*{.5}}
\put(32,17){\circle*{.5}}
\put(29,14){\circle*{.5}}
\put(26,11){\circle*{.5}}
\put(32,11){\circle*{.5}}
\put(27.5,15.5){\vector(1,-1){.07}}\multiput(26,17)(.03370787,-.03370787){89}{\line(0,-1){.03370787}}
\multiput(32,17)(-.03370787,-.03370787){89}{\line(0,-1){.03370787}}
\put(29,14){\line(0,1){0}}
\put(30.5,12.5){\vector(1,-1){.07}}\multiput(29,14)(.03370787,-.03370787){89}{\line(0,-1){.03370787}}
\put(21,14){\vector(1,0){4}}
\put(8,7){\vector(1,0){.07}}\put(5,7){\line(1,0){6}}
\put(11,4){\vector(0,-1){.07}}\put(11,7){\line(0,-1){6}}
\put(5,7){\line(0,-1){6}}
\multiput(4.93,.93)(.6,.6){11}{{\rule{.8pt}{.8pt}}}
\put(5,7.5){\makebox(0,0)[cc]{\tiny$u_1$}}
\put(11,7.5){\makebox(0,0)[cc]{\tiny$u_2$}}
\put(11,0.5){\makebox(0,0)[cc]{\tiny$u_3$}}
\put(5,0.5){\makebox(0,0)[cc]{\tiny$u_4$}}
\put(8,5){\makebox(0,0)[cc]{\tiny$v$}}
\put(17,7){\vector(1,0){.07}}\put(14,7){\line(1,0){6}}
\put(17,1){\vector(-1,0){.07}}\put(20,1){\line(-1,0){6}}
\put(14,7){\line(0,-1){6}}
\put(20,7){\line(0,-1){6}}
\multiput(19.93,.93)(0,0){3}{{\rule{.8pt}{.8pt}}}
\put(14,7.5){\makebox(0,0)[cc]{\tiny$u_1$}}
\put(20,7.5){\makebox(0,0)[cc]{\tiny$u_2$}}
\put(20,0.5){\makebox(0,0)[cc]{\tiny$u_3$}}
\put(14,0.5){\makebox(0,0)[cc]{\tiny$u_4$}}
\put(17,5){\makebox(0,0)[cc]{\tiny$v$}}
\put(26,7){\vector(1,0){.07}}\put(23,7){\line(1,0){6}}
\put(23,7){\line(0,-1){6}}
\put(23,1){\line(1,0){6}}
\put(23,7.5){\makebox(0,0)[cc]{\tiny$u_1$}}
\put(29,7.5){\makebox(0,0)[cc]{\tiny$u_2$}}
\put(29,0.5){\makebox(0,0)[cc]{\tiny$u_3$}}
\put(23,0.5){\makebox(0,0)[cc]{\tiny$u_4$}}
\put(26,5){\makebox(0,0)[cc]{\tiny$v$}}
\put(6.5,5.5){\vector(1,-1){.07}}\multiput(5,7)(.03370787,-.03370787){89}{\line(0,-1){.03370787}}
\put(9.5,2.5){\vector(1,-1){.07}}\multiput(8,4)(.03370787,-.03370787){89}{\line(0,-1){.03370787}}
\put(8,1){\vector(1,0){.07}}\put(5,1){\line(1,0){6}}
\put(5,4){\vector(0,-1){.07}}\put(5,7){\line(0,-1){6}}
\put(5,1){\line(1,1){6}}
\put(15.5,5.5){\vector(1,-1){.07}}\multiput(14,7)(.03370787,-.03370787){89}{\line(0,-1){.03370787}}
\put(18.5,2.5){\vector(-1,1){.07}}\multiput(20,1)(-.03370787,.03370787){89}{\line(0,1){.03370787}}
\put(14,4){\vector(0,-1){.07}}\put(14,7){\line(0,-1){6}}
\put(20,4){\vector(0,1){.07}}\put(20,1){\line(0,1){6}}
\put(14,1){\line(1,1){6}}
\put(29,1){\line(-1,0){6}}
\put(29,4){\vector(0,1){.07}}\put(29,1){\line(0,1){6}}
\put(26,1){\vector(-1,0){.07}}\put(29,1){\line(-1,0){6}}
\put(23,4){\vector(0,-1){.07}}\put(23,7){\line(0,-1){6}}
\put(24.5,2.5){\vector(-1,-1){.07}}\multiput(26,4)(-.03370787,-.03370787){89}{\line(0,-1){.03370787}}
\put(27.5,5.5){\vector(1,1){.07}}\multiput(26,4)(.03370787,.03370787){89}{\line(0,1){.03370787}}
\put(23,7){\line(1,-1){6}}
\put(5,7){\circle*{.5}}
\put(5,1){\circle*{.5}}
\put(11,1){\circle*{.5}}
\put(11,7){\circle*{.5}}
\put(8,4){\circle*{.5}}
\put(17,4){\circle*{.5}}
\put(14,7){\circle*{.5}}
\put(20,7){\circle*{.5}}
\put(14,1){\circle*{.5}}
\put(20,1){\circle*{.5}}
\put(23,7){\circle*{.5}}
\put(23,1){\circle*{.5}}
\put(29,1){\circle*{.5}}
\put(29,7){\circle*{.5}}
\put(4,20){\makebox(0,0)[cc]{\scriptsize$K_1\nabla K_{2,2}$}}
\put(13,20){\makebox(0,0)[cc]{\scriptsize$M_{K_{2,2}}=C_4^3$}}
\put(22,20){\makebox(0,0)[cc]{\scriptsize$M_{K_{2,2}}=C_4^1$}}
\put(31,20){\makebox(0,0)[cc]{\scriptsize$M_{K_{2,2}}=C_4^2$}}
\put(17,10){\makebox(0,0)[cc]{\scriptsize(1)}}
\put(29,10){\makebox(0,0)[cc]{\scriptsize(2)}}
\put(8,0){\makebox(0,0)[cc]{\scriptsize$M_1$}}
\put(17,0){\makebox(0,0)[cc]{\scriptsize$M_2$}}
\put(26,0){\makebox(0,0)[cc]{\scriptsize$M_3$}}
\put(26,4){\circle*{.5}}
\end{picture}

\end{center}
\caption{The graphs used in Lemma \ref{lem-f-11}}
\label{fig-x-1}
\end{figure}
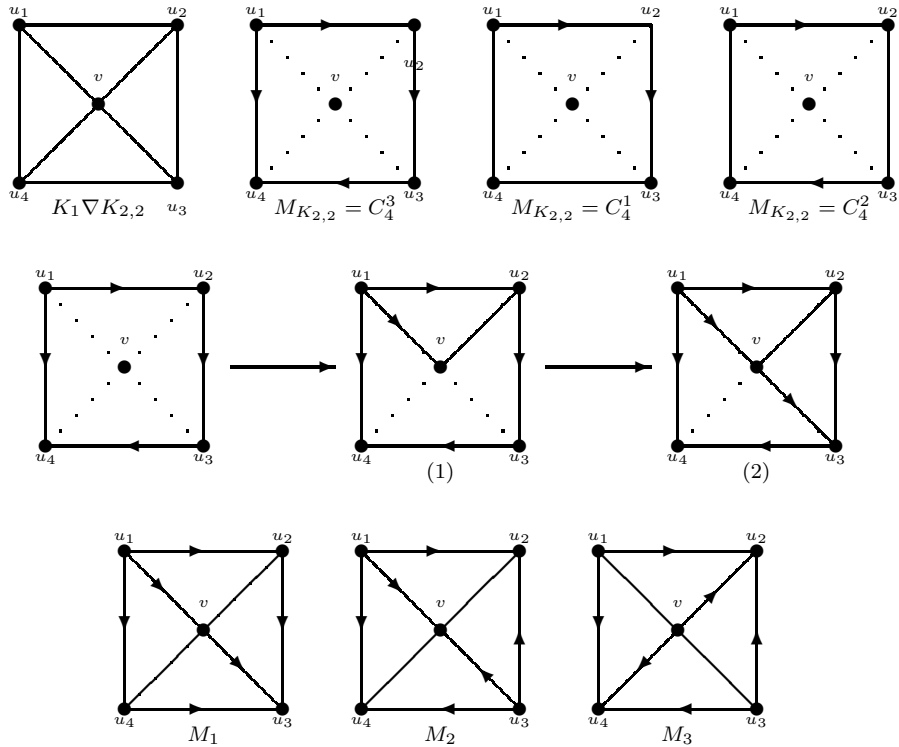

By applying Theorem \ref{lem-x-1}, we get the following result.
\begin{lemma}\label{lem-f-11}
Let $M_G$ be a connected mixed graph on $n$ vertices. If $\lambda_n(M_G)>-\frac{1+\sqrt{5}}{2}$ then $G$ is
$\{P_4,K_{1,3},K_{2,3},2K_1\nabla K_{1,2}=K_1\nabla K_{2,2},K_2\nabla 3K_1,K_2\nabla (K_2\cup K_1),K_2\nabla
K_{1,2},2K_1\nabla K_3\}$-free.
\end{lemma}
\begin{proof}
By Corollary \ref{cor-f-2}, we have $\lambda_4(M_{P_4})=\lambda_4(P_4)=-\frac{1+\sqrt{5}}{2}$ for any
$M_{P_4}\in\mathcal{M}_{P_4}$ and $\lambda_4(M_{K_{1,3}})=\lambda_4(K_{1,3})=-1.73<-\frac{1+\sqrt{5}}{2}$ for any
$M_{K_{1,3}}\in\mathcal{M}_{K_{1,3}}$. Thus, Corollary \ref{cor-f-1} implies that $G$ is $\{P_4,K_{1,3}\}$-free. Suppose to
the contrary that $G$ contains an induced $K_{2,3}$. Corollary \ref{cor-f-1} indicate that
$\lambda_5(M_{K_{2,3}})>-\frac{1+\sqrt{5}}{2}$, which contradicts Lemma \ref{lem-f-10}.

Suppose to the contrary that $G$ contains an induced $K_1\nabla K_{2,2}$ labelled as Fig.\ref{fig-x-1}. Since
$M=M_{K_1\nabla
K_{2,2}}$ has smallest eigenvalue greater than $-\frac{\sqrt{5}+1}{2}$, Lemma \ref{lem-f-6} implies all mixed triangles of
$M$ belong to $\mathcal{C}_3$ and Lemma \ref{lem-x-2} implies all quadrangles of $M$ belong to
$\mathcal{C}_4$. If the mixed induced quadrangle $M_{K_{2,2}}=C_4^3$, then we have either $u_1\rightarrow v$ or
$v\rightarrow u_2$ since $M[u_1,v,u_2]\in \mathcal{C}_3$. It leads to that $u_1\rightarrow v$ and thus
$v\leftrightarrow u_2$ since otherwise $v\rightarrow u_2$ and $M[v,u_2,u_3]\not\in\mathcal{C}_3$ (see
Fig.\ref{fig-x-1}(1)). Since $v\leftrightarrow u_2$, $u_2\rightarrow u_3$ and $M[v,u_2,u_3]\in\mathcal{C}_3$,
we have $v\rightarrow u_3$ (see Fig.\ref{fig-x-1}(2)). However, $M[v,u_3,u_4]$ cannot belong to
$\mathcal{C}_3$, a contradiction. Similarly, if the mixed induced quadrangle $M_{K_{2,2}}=C_4^1$ or $C_4^2$,
then $M\in\{M_1,M_2,M_3\}$ whose smallest eigenvalues are all $-2<-\frac{\sqrt{5}+1}{2}$, a contradiction.

Suppose to the contrary that $G$ contains an induced subgraph $H$ in $\{K_2\nabla 3K_1,K_2\nabla (K_2\cup K_1),K_2\nabla
K_{1,2},2K_1\nabla K_3\}$. Therefore, $M_G$ contains a mixed induced graph $M_H$ with order $m$. Corollary \ref{cor-f-1}
indicates that $\lambda_m(M_H)\ge \lambda_n(M_G)>-\frac{1+\sqrt{5}}{2}$. Thus, each mixed triangle in $M_H$ belongs to
$\mathcal{C}_3$. Note that $H$ contains no cycle of length greater than $3$. Theorem \ref{lem-x-1} implies that
$\lambda_{m}(M_H)=\lambda_m(H)$. It leads to a contradiction since $\lambda_5(K_2\nabla 3K_1)=-2$,
$\lambda_5(K_2\nabla(K_2\cup K_1))=-1.68$, $\lambda_5(K_2\nabla K_{1,2})=-1.65$ and $\lambda_5(2K_1\nabla K_3)=-1.65$ which
are all smaller than $-\frac{1+\sqrt{5}}{2}$.
\end{proof}
From Lemma \ref{lem-f-11}, we determine the underlying graphs of $M_G$ with smallest eigenvalue greater than
$-\frac{1+\sqrt{5}}{2}$.
\begin{lemma}\label{lem-f-12}
Let $M_G$ be a connected mixed graph on $n$ vertices. If $\lambda_n(M_G)> -\frac{1+\sqrt{5}}{2}$, then $G$ belongs to
\[
\left\{ K_{2,2}, K_1\nabla K_{1,2}, 2K_2\nabla 2K_1, (K_2\cup K_1)\nabla 2K_1 \right\}\bigcup \left\{(K_s \cup
K_t)\nabla K_1\mid s,t\ge 0, s+t=n-1\right\}.
\]
\end{lemma}
\begin{proof}
We may assume that $n\ge 2$ since there is nothing to prove when $n=1$. From Lemma \ref{lem-f-11}, we have $G$ is $P_4$-free
and thus $G=X\bigtriangledown Y$ with $|X|,|Y|\ge 1$ due to Lemma \ref{lem-f-3}. If both $X$ and $Y$ have no edge, then
$G= K_{m,n}$ and thus $G\in\{K_2, K_{1,2},K_{2,2}\}$ due to Lemma \ref{lem-f-10}, where both $K_2=(K_1\cup K_0)\nabla K_1$
and $K_{1,2}=(K_1\cup K_1)\nabla K_1$ have the form $(K_s\cup K_t)\nabla K_1$. Now we may assume that one of $X$ and $Y$
contains $K_2$, say $X$.
Therefore, Lemma \ref{lem-f-11} implies that $Y$ is $\{3K_1, K_2\cup K_1, K_{1,2}\}$-free and thus $Y\in\{2K_1,K_s\mid
s\geq1\}$ due to Lemma \ref{lem-f-4}.  If $Y=K_s $ with $s\geq2$, then Lemma \ref{lem-f-11} implies that $X$ is $\{3K_1,
K_2\cup K_1, K_{1,2}\}$-free. Thus, Lemma \ref{lem-f-4} means that $X=K_r$ with $r\ge 2$ since $X$ contains $K_2$.
Therefore, $G=K_n=(K_{n-1}\cup K_0)\nabla K_1$ with $n\ge 4$. If $Y= 2K_1$, then Lemma \ref{lem-f-11} indicates that $X$
is $\{ 3K_1, K_{1,2}, K_3\}$-free. Hence, $X\in\{2K_2,K_2\cup K_1,K_2\}$ due to Lemma \ref{lem-f-4}, and thus $G\in\{
2K_2\bigtriangledown 2K_1,(K_2\cup K_1)\bigtriangledown 2K_1, K_2\bigtriangledown 2K_1=K_1\nabla K_{1,2}\}$.

In what follows, we consider the case of $Y= K_1$, that is $G=X\bigtriangledown K_1$.
Since $G$ is $K_{1,3}$-free according to Lemma \ref{lem-f-11}, we have $X$ is $\{3K_1\}$-free and $X$ has at most two
connected components.
Suppose that $X$ has two connected component, say $X_1$ and $X_2$ with $|X_1|,|X_2|\geq 1$.
Then both $X_1$ and $ X_2$ are $P_3$-free since otherwise $X$ has an induced $3K_1$, and so $X_1$ and $ X_2$ are complete
graphs. Therefore, $G=(K_s\cup K_t)\bigtriangledown K_1$ with $s+t=n-1$ and $s,t\geq1$.
Next we may assume that $X$ is connected. Since $X$ is $P_4$-free, we have $X=X_1\bigtriangledown Y_1$ with $|X_1|,
|Y_1|\geq1$ from Lemma \ref{lem-f-3}.
If both $X_1$ and $Y_1$ have no edge, then $X$ is a bipartite graph and so $X\in\{ K_2, K_{1,2}, K_{2,2}\}$ by
Lemma \ref{lem-f-10}. Note that Lemma \ref{lem-f-11} means that $G$ is $K_1\nabla K_{2,2}$-free. Thus, $G\in\{K_3,
K_1\bigtriangledown K_{1,2}\}$.
Now we may assume $X_1$ contains a $K_2$. Then $Y_1$ is a $\{3K_1, K_2\cup K_1, K_{1,2}\}$-free by Lemma \ref{lem-f-11}.
Hence, $Y_1\in \{ K_s, 2K_1\}$ by Lemma \ref{lem-f-4}. If $Y_1= K_s$ with $s\geq2$, then $X_1$ is $\{3K_1, K_2\cup K_1,
K_{1,2}\}$-free by Lemma \ref{lem-f-11}.
By Lemma \ref{lem-f-4}, we have $X_1= K_t (t\geq2)$ since $X_1$ has an edge.
Note that $G=X \bigtriangledown K_1=(X_1\bigtriangledown Y_1)\bigtriangledown K_1$.
Therefore, $G= (K_s \bigtriangledown K_t)\bigtriangledown K_1=K_n$ for $n\geq5$. If $Y_1= 2K_1$, then $G=
X_1\bigtriangledown 2K_1 \bigtriangledown K_1=X_1\bigtriangledown
K_{1,2}$.
Lemma \ref{lem-f-11} indicates that $X_1$ is $\{2K_1, K_2\}$-free, and thus $X_1= K_1$. Therefore,
$G= K_1\bigtriangledown K_{1,2}$. If $Y_1= K_1$, then $G=X_1\bigtriangledown K_1\bigtriangledown
K_1=X_1\bigtriangledown K_2$.
Therefore $X_1$ is $\{3K_1, K_2\cup K_1, K_{1,2}\}$-free by Lemma \ref{lem-f-11}.
Lemma \ref{lem-f-4} indicates that $X_1= K_s (s\geq2)$ since $X_1$ has an edge.
Thus, $G=K_s\bigtriangledown K_2=(K_{n-1}\nabla K_0)\nabla K_1$ with $n\geq4$.

The proof is completed.
\end{proof}

In what follows, we completely determine $M_G$ with smallest eigenvalue greater than $-\frac{1+\sqrt{5}}{2}$ by  considering
its underlying graphs as given in Lemma \ref{lem-f-12} one by one.
\begin{lemma}\label{2K2jion}
Let $M_{G}$ be a mixed graph with $G=(K_2\cup
K_1)\nabla2K_1$. If
any mixed triangle of $M_G$ belongs to $\mathcal{C}_3$ and any induced mixed quadrangle of $M_G$ belongs to $\mathcal{C}_4$, then $ M_{G}\in \{H_{1},...,H_{9}\}$ shown in the Appendix.
\end{lemma}
\begin{proof}
Let $V(G)=\{v_1,v_2,...,v_5\}$ (see the Appendix). Clearly, $M_G$ has two induced
quadrangles $M_G[v_1,v_2,v_3,v_4]$ and $M_G[v_1,v_5,v_3,v_4]$. Note that any induced mixed quadrangle belongs to
$\mathcal{C}_4$. We divide  four cases to discuss.

\textbf{Case 1. one of them is $C_4^3$.}

In this case, we may assume $M_G[v_1,v_2,v_3,v_4]=C_4^3$. Clearly, there are two different orientations of the mixed cycle
$M_G[v_1,v_2,v_3,v_4]$, that is, $v_1\rightarrow v_2\rightarrow v_3\rightarrow v_4$ and $v_1\rightarrow v_4$, or
$v_2\rightarrow v_3\rightarrow v_4\rightarrow  v_1$ and $v_2\rightarrow v_1$. If the former happens, then it must holds that
$M_G[v_1,v_5,v_3,v_4]=C_4^3$, ant thus $v_1\rightarrow v_5$, $v_5\rightarrow v_3$ and $v_2\leftrightarrow v_5$ since any
induced  mixed triangle  belongs to $\mathcal{C}_3$. It yields that $M_G= H_{1}$. If the latter happens, then
then $M_G[v_1,v_5,v_3,v_4]=C_4^1$ or $C_4^3$. Therefore, one can easily verify that $M_G= H_{2}$ when
$M_G[v_1,v_5,v_3,v_4]=C_4^1$, and $M_G=H_{3}$ when $M_G[v_1,v_5,v_3,v_4]=C_4^3$ similarly.

\textbf{Case 2. $M_G[v_1,v_2,v_3,v_4]= M_G[v_1,v_5,v_3,v_4]=C_4^2$}.

In this case, there are also two different orientations of $M_G[v_1,v_2,v_3,v_4]$, that is $v_1\rightarrow v_2$,
$v_2\leftrightarrow v_3$, $v_3\rightarrow v_4$ and $v_4\leftrightarrow v_1$, or $v_1\leftrightarrow v_2$, $v_2\rightarrow
v_3$, $v_3\leftrightarrow v_4$ and $v_4\rightarrow v_1$. Therefore, one can easily verify that $M_G= H_{4}$ when the former
happens and $M_G= H_{5}$ when the latter happens by noticing that any mixed triangle  belongs to $\mathcal{C}_3$.

\textbf{Case 3: $M_G[v_1,v_2,v_3,v_4]= M_G[v_1,v_5,v_3,v_4]=C_4^1$}.

In this case, there are three different orientations of $M_G[v_1,v_2,v_3,v_4]$, that is,  $v_1\rightarrow v_2$,
$v_2\rightarrow v_3$, $v_3\leftrightarrow v_4$ and $v_4\leftrightarrow v_1$, or $v_1\leftrightarrow v_2$, $v_2\rightarrow
v_3$, $v_3\rightarrow v_4$ and $v_4\leftrightarrow v_1$, or $v_1\leftrightarrow v_2$, $v_2\leftrightarrow v_3$,
$v_3\rightarrow v_4$ and $v_4\rightarrow v_1$. One can easily verify that $M_G= H_{6}$ when the first case happens, $M_G=
H_{7}$ when the second case happens, and $M_G= H_{8}$ when the third case happens.

\textbf{Case 4: $M_G[v_1,v_2,v_3,v_4]= C_4^1$ and $ M_G[v_1,v_5,v_3,v_4]= C_4^2$}.

In this case, there are three different orientations of $M_G[v_1,v_2,v_3,v_4]$, that is,  $v_1\rightarrow v_2$,
$v_2\rightarrow v_3$, $v_3\leftrightarrow v_4$ and $v_4\leftrightarrow v_1$, or $v_1\leftrightarrow v_2$, $v_2\rightarrow
v_3$, $v_3\rightarrow v_4$ and $v_4\leftrightarrow v_1$, or $v_1\leftrightarrow v_2$, $v_2\leftrightarrow v_3$,
$v_3\rightarrow v_4$ and $v_4\rightarrow v_1$. If the first or the third case happens, then $M_G[v_1,v_5,v_3,v_4]$ cannot be
$C_4^2$, which is impossible. If the second case happens, then $M_G= H_{9}$.

This completes the proof.
\end{proof}

As similar to Lemma \ref{2K2jion}, we present the following result but omit the tautological proof.
\begin{lemma}\label{K2cupK1}
Let $M_{G}$ be a mixed graph with $G=2K_2\nabla2K_1$. If any mixed triangle of $M_G$ belongs to $\mathcal{C}_3$ and any induced mixed quadrangle
belongs to $\mathcal{C}_4$, then $ M_{G}\in \{H_{10},...,H_{20}\}$ shown in the Appendix.
\end{lemma}

 The {\it coalescence} $M\bullet_{u,v} M'$ of two mixed graphs $M$ and $M'$ is obtained from
$M\cup M'$ by identifying a vertex $u$ of $M$ with a vertex $v$ of $M'$.

\begin{lemma}\label{lem-x-3}
Let $G$ be a connected graph with a cut vertex $v$ such that $G-v=G_1\cup G_2$ with $V_1=V(G_1)$ and $V_2=V(G_2)$. If
$G_1^+=G[V_1\cup \{v\}]$ and $G_2^+=G[V_2\cup\{v\}]$, then $[G]=\{M\bullet_{v,v} M'\mid M\in [M_{G_1^+}], M'\in[G_2^+]\}$.
\end{lemma}
\begin{proof}
It is clear that $G=G_1^+\bullet_{v,v}G_2^+$. For any $M\bullet M'$ with $M\in [M_{G_1^+}]$ and $M'\in[G_2^+]$, there exist
diagonal matrices $D_1$ and $D_2$ with diagonal entries in $\{\pm 1,\pm i\}$ such that $D_1H(M)D_1^*=H(G_1^+)$ and
$D_2H(M')D_2=H(G_2^+)$. Note that the $v$-th diagonal entries of $D_1$ and $D_2$ satisfy $D_2(v)=\epsilon D_1(v)$ for some
$\epsilon\in\{\pm 1,\pm i\}$. Let $D$ be the diagonal matrix indexed by $V(G)$ such that the diagonal entries are
$D(v_1)=D_1(v_1)$ for $v_1\in V_1\cup\{v\}$ and $D(v_2)=\epsilon D_2(v_2)$ for $v_2\in V_2$. Therefore,  one can easily
verify that $DH(M\bullet_{v,v}M')D^*=H(G_1^+\bullet_{v,v}G_2^+)=H(G)$, and thus $M\bullet_{v,v}M'\in[G]$.

Conversely, for any $M_G\in[G]$, there exists diagnal matrix $D$ such that $DH(M_G)D^*=H(G)$. Note that $M_G=M\bullet_{v,v}
M'$ where $M=M_G[V_1\cup \{v\}]$ and $M'=M_G[V_2\cup\{v\}]$. Let $D_1$ and $D_2$ be the diagonal matrices indexed by
$V_1\cup \{v\}$ and $V_2\cup \{v\}$ respectively such that the diagonal entries are $D_1(v_1)=D(v_1)$  for and $v_1\in
V_1\cup \{v\}$ and $D_2(v_2)=D(v_2)$ for any $v_2\in V_2\cup \{v\}$. Therefore, one can easily verify that
$D_1H(M)D_1^*=H(G_1^+)$ and $D_2H(M')D_2^*=H(G_2^+)$, and thus $M\in[G_1^+]$ and $M'\in[G_2^+]$.
\end{proof}

Now we are ready to present our main result.

\begin{theorem}\label{thm-f-4}
Let $M_G$ be a connected mixed graph on $n$ vertices. Then $\lambda_n> -\frac{1+\sqrt{5}}{2}$ if and only if
$M_G\in\mathcal{H}_1\cup \mathcal{H}_2\cup\mathcal{H}_3\cup\mathcal{H}_4$, where
\[\left\{\begin{array}{l}
\mathcal{H}_1=\{C_4^1, C_4^2, C_4^3,H_1,H_2,\ldots, H_{27}\},\\[2mm] \mathcal{H}_2=\{ M\bullet_{u,v} M'\mid u\in V(M),v\in
V(M'), M\in[K_3],M'\in[K_3]\cup
[K_4] \},\\[2mm]
\mathcal{H}_3=[K_n]=\{K_n[s,t]\mid s,t\ge0, s+t=n\},\\[2mm]
\mathcal{H}_4=\{M\bullet_{u,v} M'\mid u\in V(M),v\in V(M'), M\in[K_2],M'\in[K_{n-1}]\}.
\end{array}\right.\]
\end{theorem}
\begin{proof}
To prove the sufficiency, it only needs to show that each graph in
$\mathcal{H}_1\cup\mathcal{H}_2\cup\mathcal{H}_3\cup\mathcal{H}_4$ has smallest eigenvalue greater than
$-\frac{1+\sqrt{5}}{2}$. By immediate calculations, the smallest eigenvalues of $C_4^1,C_4^2,C_4^3$ are all $-\sqrt{2}>
-\frac{1+\sqrt{5}}{2}$ and the smallest eigenvalues of $H_1, H_2,\ldots,H_{27}$ are all $-1.56> -\frac{1+\sqrt{5}}{2}$ (see
the Appendix). For any
$M\in[K_3]$ and $M'\in[K_3]$, Lemma \ref{lem-x-3} implies that $\operatorname{Sp}(M\bullet_{u,v}
M')=\operatorname{Sp}(K_3\bullet_{u,v}
K_3)$. Thus, we have $\operatorname{Sp}(M\bullet_{u,v} M')=\{2.56,1,[-1]^2,-1.56\}$ by immediate calculations. Similarly, if
$M\in
[K_3]$ and $M'\in [K_4]$, we have $\operatorname{Sp}(M\bullet_{u,v} M')=\{3.26,1.34,[-1]^3,-1.60\}$.  Theorem \ref{thm-f-1}
implies that $K_n[s,t]$ has smallest eigenvalue $-1$. For any $M_G\in\mathcal{H}_4$, Lemma \ref{lem-x-3} implies that
$\operatorname{Sp}(M_G)=\operatorname{Sp}(K_{n-1}\bullet_{u,v} K_2)$, whose smallest eigenvalue are the smallest root of
$\varphi(x)=x^3 + (3 - n)x^2 + (1 - n)x - 1=0$. Note that $\varphi(-1)=0$, $\varphi(-\frac{1+\sqrt{5}}{2})=1-n<0$ for $n\geq
2$. The smallest root of $\varphi(x)$ is greater than $-\frac{1+\sqrt{5}}{2}$ by the image of $\varphi(x)$, and thus
$\lambda_n(M_G)> -\frac{1+\sqrt{5}}{2}$.

In what follows, we show the necessity.
Since $\lambda_n(M_G)> -\frac{1+\sqrt{5}}{2}$, Lemmas \ref{lem-f-6} and \ref{lem-x-2} indicate that
any mixed triangle of $M_G$ belongs to $\mathcal{C}_3$ and any mixed induced quadrangle of $M_{G}$ belongs to
$\mathcal{C}_4$.
From Lemma \ref{lem-f-12}, the underlying graph $G$ belongs to
\[
\left\{ K_{2,2}, K_1\nabla K_{1,2}, 2K_2\nabla 2K_1, (K_2\cup K_1)\nabla 2K_1 \right\}\bigcup \left\{(K_s \cup
K_t)\nabla K_1\mid s+t=n-1\right\}.
\]

If $G= K_{2,2}$,   then $M_G\in\{ C_4^1, C_4^2, C_4^3\}\subseteq \mathcal{H}_1$ due to Lemma \ref{lem-x-2}. If $G= K_1\nabla
K_{1,2}$, then $G$ contains no induced cycle with length greater than $3$.
Thus, Theorem \ref{lem-x-1} implies that $M_G\in[K_1\nabla K_{1,2}]=\{H_{21},\ldots, H_{27}\}\subseteq\mathcal{H}_1$.
If $G=2K_2\nabla 2K_1$ or $(K_2\cup K_1)\nabla 2K_1$, then $M_G\in\{H_1,\ldots,H_{20}\}\subseteq\mathcal{H}_1$ due to Lemmas
\ref{2K2jion} and \ref{K2cupK1}.

If $G=(K_s \cup K_t)\nabla K_1$ with $s=0$ or $t=0$, then $G= K_n$. Since any mixed triangle of $M_G$ belongs to
$\mathcal{C}_3$, Theorem \ref{thm-f-1} means that $M_{G}= M_{K_n}\in\{K_{n}[s,t]\mid s,t\ge 0, s+t=n\}=[K_n]=\mathcal{H}_3$.
Now we suppose $G=(K_s \cup K_t)\nabla K_1$ with $s,t\ge 1$ and $s\ge t$. Note that $G$ contains no induced cycle with
length greater than $3$, Theorem \ref{lem-x-1} indicates that $M_{G}\in [G]$ and thus $\operatorname{Sp}(M_G)
=\operatorname{Sp}(G)$. Note that $[(K_s \cup K_t)\nabla K_1]=\{M\bullet_{u,v} M'\mid M\in [K_{s+1}], M'\in[K_{t+1}]\}$ due
to Lemma \ref{lem-x-3}. Assume that $\pi$: $V(G)=V_1\cup\{v\}\cup V_2 $ is the partition such that $G[V_1\cup\{v\}
]=K_{s+1}$ and $G[V_2\cup\{v\} ]=K_{t+1}$. The Hermitian matrix of $G$ is
\[H(G)=\begin{pmatrix}J_s-I_s&\mathbf{1}_s&\mathbf{0}_{s\times t}\\\mathbf{1}^T_s&\mathbf{0}_{t\times
s}&\mathbf{1}_t^T\\0&\mathbf{1}&J_s-I_s\end{pmatrix},\]
where $J$, $I$, $\mathbf{1}$ and $\mathbf{0}$ are respectively the all-one matrix, identity matrix, all-one vector and zero
matrix with the corresponding size. Therefore, Lemma \ref{lem-f-2} indicates that $\pi$ is an equitable partition with
quotient matrix
\[H_{\pi}=\begin{pmatrix}
s-1 &1&0\\s&0&t\\0 &1&t-1
\end{pmatrix}.
\]
Assume that $V_1=\{v_1,v_2,\ldots,v_s\}$ and $V_2=\{u_1,u_2,\ldots,u_t\}$. For $1\le j\le s$ and $1\le k\le t$, let
$\delta_{1,j}\in\mathbb{R}^s$ be the vector indexed by $V_1$ such that $\delta_{1,j}(v_1)=1$, $\delta_{1,j}(v_j)=-1$ and
$\delta_{1,j}(v_{j'})=0$ for $j'\not\in\{1,j\}$ and let $\delta_{2,k}$ be the vector indexed by $V_2$ such that
$\delta_{2,k}(u_1)=1$, $\delta_{2,k}(u_k)=-1$ and $\delta_{2,k}(u_{k'})=0$ for $k'\not\in\{1,k\}$. It is easy to see that
$H(G)\delta_{1,j}=-\delta_{1,j}$ and $H(G)\delta_{2,k}=-\delta_{2,k}$ for any $j$ and $k$, and thus $H$ has an eigenvalue
$-1$ with multiplicity at least $s+t-2=n-3$. Lemma \ref{lem-f-2} implies that the other three eigenvalues of $G$ are just
the roots $\epsilon_1\ge\epsilon_2\ge\epsilon_3$ of the function $f(x)=det(xI-B_{\pi})=x^3 + (2-t -s)x^2 + (st -2t - 2s +
1)x - s - t + 2st$, and thus $\epsilon_3=\lambda_n(G)>-\frac{1+\sqrt{5}}{2}$. It is clear that $f(0)=st-s-t\ge0$. Note that
$\epsilon_1>0$. By the image of the function $f(x)$, we have $f(-\frac{1+\sqrt{5}}{2})<0$. If $t\ge 3$ then
\[f(-\frac{1+\sqrt{5}}{2})=\frac{3-\sqrt{5}}{2}(st-s-t)+\frac{1-\sqrt{5}}{2}\ge\frac{3-\sqrt{5}}{2}s-\frac{1-\sqrt{5}}{2}\ge5-2\sqrt{5}>0,\]
a contradiction. Thus, we have $t\le 2$. If $t=2$ then
$f(-\frac{1+\sqrt{5}}{2})=\frac{3-\sqrt{5}}{2}s-\frac{5-\sqrt{5}}{2}<0$. It leads to $s<\frac{5+\sqrt{5}}{2}\approx 3.62$.
Thus, we have $s=2$ or $3$ since $s\ge t=2$. It means $M_G\in[(K_2\cup K_2)\nabla K_1]\cup [(K_3\cup K_2)\nabla
K_1]=\{M\bullet_{u,v} M'\mid M\in [K_3], M'\in [K_3]\cup[K_4]\}=\mathcal{H}_2$. If $t=1$ then
$f(-\frac{1+\sqrt{5}}{2})=-1<0$ always holds. Thus, $s\ge t=1$ and $M_G\in[(K_s\cup K_1)\nabla K_1]=\{M\bullet_{u,v} M'\mid
M\in[K_2],M'\in[K_{n-1}]\}=\mathcal{H}_4$.

This completes the proof.
\end{proof}

\section*{Acknowledgments}
The first author is supported by National Natural Science Foundation of China (No. 12001544).
The second author is supported by National Natural Science Foundation of China (Nos. 12061074, 11701492) and the China
Postdoctoral Science Foundation (No. 2019M661398). The third author is supported by
National Natural Science Foundation of China (No. 11671344).

\newpage
\section*{Appendix: The mixed graphs $\mathcal{C}_4$ and $H_1,\ldots,H_{27}$ with their smallest eigenvalues.}
\begin{figure}[htbp]
\begin{center}
\unitlength 3.500mm 
\linethickness{0.4pt}
\ifx\plotpoint\undefined\newsavebox{\plotpoint}\fi 
\begin{picture}(96.136,57.554)(0,0)
\thicklines
\put(95.886,.25){\line(0,1){0}}
\put(95.886,.25){\circle*{.5}}
\put(95.886,3.25){\line(0,1){0}}
\put(95.886,6.25){\circle*{.5}}
\put(2.25,48.304){\line(1,0){4}}
\put(4.25,41.304){\line(4,3){4}}
\multiput(8.25,44.304)(-.03333333,.06666667){60}{\line(0,1){.06666667}}
\multiput(2.25,48.304)(-.03333333,-.06666667){60}{\line(0,-1){.06666667}}
\put(.25,44.304){\line(4,-3){4}}
\put(2.25,48.304){\line(3,-2){6}}
\put(6.25,48.304){\line(-3,-2){6}}
\put(2.25,48.304){\circle*{.5}}
\put(6.25,48.304){\circle*{.5}}
\put(.25,44.304){\circle*{.5}}
\put(8.25,44.304){\circle*{.5}}
\put(11.25,48.304){\line(1,0){4}}
\put(13.25,41.304){\line(4,3){4}}
\multiput(17.25,44.304)(-.03333333,.06666667){60}{\line(0,1){.06666667}}
\multiput(11.25,48.304)(-.03333333,-.06666667){60}{\line(0,-1){.06666667}}
\put(11.25,48.304){\line(3,-2){6}}
\put(15.25,48.304){\line(-3,-2){6}}
\put(11.25,48.304){\circle*{.5}}
\put(15.25,48.304){\circle*{.5}}
\put(9.25,44.304){\circle*{.5}}
\put(17.25,44.304){\circle*{.5}}
\put(20.25,48.304){\line(1,0){4}}
\put(22.25,41.304){\line(4,3){4}}
\multiput(26.25,44.304)(-.03333333,.06666667){60}{\line(0,1){.06666667}}
\multiput(20.25,48.304)(-.03333333,-.06666667){60}{\line(0,-1){.06666667}}
\put(18.25,44.304){\line(4,-3){4}}
\put(20.25,48.304){\line(3,-2){6}}
\put(24.25,48.304){\line(-3,-2){6}}
\put(20.25,48.304){\circle*{.5}}
\put(24.25,48.304){\circle*{.5}}
\put(18.25,44.304){\circle*{.5}}
\put(26.25,44.304){\circle*{.5}}
\put(22.25,41.304){\circle*{.5}}
\put(29.25,48.304){\line(1,0){4}}
\put(31.25,41.304){\line(4,3){4}}
\multiput(35.25,44.304)(-.03333333,.06666667){60}{\line(0,1){.06666667}}
\multiput(29.25,48.304)(-.03333333,-.06666667){60}{\line(0,-1){.06666667}}
\put(27.25,44.304){\line(4,-3){4}}
\put(29.25,48.304){\line(3,-2){6}}
\put(33.25,48.304){\line(-3,-2){6}}
\put(29.25,48.304){\circle*{.5}}
\put(33.25,48.304){\circle*{.5}}
\put(27.25,44.304){\circle*{.5}}
\put(35.25,44.304){\circle*{.5}}
\put(31.25,41.304){\circle*{.5}}
\put(2.25,57.304){\line(1,0){4}}
\put(4.25,50.304){\line(4,3){4}}
\multiput(8.25,53.304)(-.03333333,.06666667){60}{\line(0,1){.06666667}}
\multiput(2.25,57.304)(-.03333333,-.06666667){60}{\line(0,-1){.06666667}}
\put(.25,53.304){\line(4,-3){4}}
\put(2.25,57.304){\line(3,-2){6}}
\put(6.25,57.304){\line(-3,-2){6}}
\put(2.25,57.304){\circle*{.5}}
\put(6.25,57.304){\circle*{.5}}
\put(.25,53.304){\circle*{.5}}
\put(8.25,53.304){\circle*{.5}}
\put(4.25,50.304){\circle*{.5}}
\put(11.25,57.304){\line(1,0){4}}
\put(13.25,50.304){\line(4,3){4}}
\multiput(17.25,53.304)(-.03333333,.06666667){60}{\line(0,1){.06666667}}
\multiput(11.25,57.304)(-.03333333,-.06666667){60}{\line(0,-1){.06666667}}
\put(9.25,53.304){\line(4,-3){4}}
\put(11.25,57.304){\line(3,-2){6}}
\put(15.25,57.304){\line(-3,-2){6}}
\put(11.25,57.304){\circle*{.5}}
\put(15.25,57.304){\circle*{.5}}
\put(9.25,53.304){\circle*{.5}}
\put(17.25,53.304){\circle*{.5}}
\put(13.25,50.304){\circle*{.5}}
\put(20.25,57.304){\line(1,0){4}}
\put(22.25,50.304){\line(4,3){4}}
\multiput(26.25,53.304)(-.03333333,.06666667){60}{\line(0,1){.06666667}}
\multiput(20.25,57.304)(-.03333333,-.06666667){60}{\line(0,-1){.06666667}}
\put(20.25,57.304){\line(3,-2){6}}
\put(24.25,57.304){\line(-3,-2){6}}
\put(20.25,57.304){\circle*{.5}}
\put(24.25,57.304){\circle*{.5}}
\put(18.25,53.304){\circle*{.5}}
\put(26.25,53.304){\circle*{.5}}
\put(22.25,50.304){\circle*{.5}}
\put(29.25,57.304){\line(1,0){4}}
\put(31.25,50.304){\line(4,3){4}}
\multiput(35.25,53.304)(-.03333333,.06666667){60}{\line(0,1){.06666667}}
\multiput(29.25,57.304)(-.03333333,-.06666667){60}{\line(0,-1){.06666667}}
\put(27.25,53.304){\line(4,-3){4}}
\put(29.25,57.304){\line(3,-2){6}}
\put(33.25,57.304){\line(-3,-2){6}}
\put(29.25,57.304){\circle*{.5}}
\put(33.25,57.304){\circle*{.5}}
\put(27.25,53.304){\circle*{.5}}
\put(35.25,53.304){\circle*{.5}}
\put(31.25,50.304){\circle*{.5}}
\put(38.25,57.304){\line(1,0){4}}
\put(40.25,50.304){\line(4,3){4}}
\multiput(44.25,53.304)(-.03333333,.06666667){60}{\line(0,1){.06666667}}
\multiput(38.25,57.304)(-.03333333,-.06666667){60}{\line(0,-1){.06666667}}
\put(36.25,53.304){\line(4,-3){4}}
\put(38.25,57.304){\line(3,-2){6}}
\put(42.25,57.304){\line(-3,-2){6}}
\put(38.25,57.304){\circle*{.5}}
\put(42.25,57.304){\circle*{.5}}
\put(36.25,53.304){\circle*{.5}}
\put(44.25,53.304){\circle*{.5}}
\put(40.25,50.304){\circle*{.5}}
\put(13.25,39.304){\line(-1,-1){4}}
\put(9.25,35.304){\line(1,0){3}}
\multiput(12.25,35.304)(.0333333,.1333333){30}{\line(0,1){.1333333}}
\multiput(13.25,39.304)(.0333333,-.1333333){30}{\line(0,-1){.1333333}}
\put(14.25,35.304){\line(1,0){3}}
\put(17.25,35.304){\line(-1,1){4}}
\put(9.25,35.304){\line(1,-1){4}}
\put(13.25,31.304){\line(1,1){4}}
\multiput(14.25,35.304)(-.0333333,-.1333333){30}{\line(0,-1){.1333333}}
\multiput(13.25,31.304)(-.0333333,.1333333){30}{\line(0,1){.1333333}}
\put(12.25,35.304){\line(0,1){0}}
\put(13.25,39.304){\circle*{.5}}
\put(9.25,35.304){\circle*{.5}}
\put(13.25,31.304){\circle*{.5}}
\put(17.25,35.304){\circle*{.5}}
\put(14.25,35.304){\circle*{.5}}
\put(12.25,35.304){\circle*{.5}}
\put(22.25,39.304){\line(-1,-1){4}}
\put(18.25,35.304){\line(1,0){3}}
\multiput(21.25,35.304)(.0333333,.1333333){30}{\line(0,1){.1333333}}
\multiput(22.25,39.304)(.0333333,-.1333333){30}{\line(0,-1){.1333333}}
\put(23.25,35.304){\line(1,0){3}}
\put(26.25,35.304){\line(-1,1){4}}
\put(18.25,35.304){\line(1,-1){4}}
\put(22.25,31.304){\line(1,1){4}}
\multiput(23.25,35.304)(-.0333333,-.1333333){30}{\line(0,-1){.1333333}}
\multiput(22.25,31.304)(-.0333333,.1333333){30}{\line(0,1){.1333333}}
\put(21.25,35.304){\line(0,1){0}}
\put(22.25,39.304){\circle*{.5}}
\put(18.25,35.304){\circle*{.5}}
\put(22.25,31.304){\circle*{.5}}
\put(26.25,35.304){\circle*{.5}}
\put(23.25,35.304){\circle*{.5}}
\put(21.25,35.304){\circle*{.5}}
\put(31.25,39.304){\line(-1,-1){4}}
\put(27.25,35.304){\line(1,0){3}}
\multiput(30.25,35.304)(.0333333,.1333333){30}{\line(0,1){.1333333}}
\multiput(31.25,39.304)(.0333333,-.1333333){30}{\line(0,-1){.1333333}}
\put(32.25,35.304){\line(1,0){3}}
\put(35.25,35.304){\line(-1,1){4}}
\put(27.25,35.304){\line(1,-1){4}}
\put(31.25,31.304){\line(1,1){4}}
\multiput(32.25,35.304)(-.0333333,-.1333333){30}{\line(0,-1){.1333333}}
\multiput(31.25,31.304)(-.0333333,.1333333){30}{\line(0,1){.1333333}}
\put(30.25,35.304){\line(0,1){0}}
\put(31.25,39.304){\circle*{.5}}
\put(27.25,35.304){\circle*{.5}}
\put(31.25,31.304){\circle*{.5}}
\put(35.25,35.304){\circle*{.5}}
\put(32.25,35.304){\circle*{.5}}
\put(30.25,35.304){\circle*{.5}}
\put(40.25,39.304){\line(-1,-1){4}}
\put(36.25,35.304){\line(1,0){3}}
\multiput(39.25,35.304)(.0333333,.1333333){30}{\line(0,1){.1333333}}
\multiput(40.25,39.304)(.0333333,-.1333333){30}{\line(0,-1){.1333333}}
\put(41.25,35.304){\line(1,0){3}}
\put(44.25,35.304){\line(-1,1){4}}
\put(36.25,35.304){\line(1,-1){4}}
\put(40.25,31.304){\line(1,1){4}}
\multiput(41.25,35.304)(-.0333333,-.1333333){30}{\line(0,-1){.1333333}}
\multiput(40.25,31.304)(-.0333333,.1333333){30}{\line(0,1){.1333333}}
\put(39.25,35.304){\line(0,1){0}}
\put(40.25,39.304){\circle*{.5}}
\put(36.25,35.304){\circle*{.5}}
\put(40.25,31.304){\circle*{.5}}
\put(44.25,35.304){\circle*{.5}}
\put(41.25,35.304){\circle*{.5}}
\put(39.25,35.304){\circle*{.5}}
\put(40.25,49.304){\line(-1,-1){4}}
\put(36.25,45.304){\line(1,0){3}}
\multiput(39.25,45.304)(.0333333,.1333333){30}{\line(0,1){.1333333}}
\multiput(40.25,49.304)(.0333333,-.1333333){30}{\line(0,-1){.1333333}}
\put(41.25,45.304){\line(1,0){3}}
\put(44.25,45.304){\line(-1,1){4}}
\put(36.25,45.304){\line(1,-1){4}}
\put(40.25,41.304){\line(1,1){4}}
\multiput(41.25,45.304)(-.0333333,-.1333333){30}{\line(0,-1){.1333333}}
\multiput(40.25,41.304)(-.0333333,.1333333){30}{\line(0,1){.1333333}}
\put(39.25,45.304){\line(0,1){0}}
\put(40.25,49.304){\circle*{.5}}
\put(36.25,45.304){\circle*{.5}}
\put(44.25,45.304){\circle*{.5}}
\put(41.25,45.304){\circle*{.5}}
\put(39.25,45.304){\circle*{.5}}
\put(1.25,55.304){\vector(-1,-2){.07}}\multiput(2.25,57.304)(-.03333333,-.06666667){60}{\line(0,-1){.06666667}}
\put(3.25,55.304){\vector(-3,-2){.07}}\multiput(6.25,57.304)(-.050420168,-.033613445){119}{\line(-1,0){.050420168}}
\put(7.25,55.304){\vector(-1,2){.07}}\multiput(8.25,53.304)(-.03333333,.06666667){60}{\line(0,1){.06666667}}
\put(5.25,55.304){\vector(-3,2){.07}}\multiput(8.25,53.304)(-.050420168,.033613445){119}{\line(-1,0){.050420168}}
\put(2.25,51.804){\vector(4,-3){.07}}\multiput(.25,53.304)(.04494382,-.03370787){89}{\line(1,0){.04494382}}
\put(6.25,51.804){\vector(-4,-3){.07}}\multiput(8.25,53.304)(-.04494382,-.03370787){89}{\line(-1,0){.04494382}}
\put(13.25,57.304){\vector(1,0){.07}}\put(11.25,57.304){\line(1,0){4}}
\put(10.25,55.304){\vector(-1,-2){.07}}\multiput(11.25,57.304)(-.03333333,-.06666667){60}{\line(0,-1){.06666667}}
\put(14.25,55.304){\vector(3,-2){.07}}\multiput(11.25,57.304)(.050420168,-.033613445){119}{\line(1,0){.050420168}}
\put(11.25,51.804){\vector(4,-3){.07}}\multiput(9.25,53.304)(.04494382,-.03370787){89}{\line(1,0){.04494382}}
\put(15.25,51.804){\vector(4,3){.07}}\multiput(13.25,50.304)(.04494382,.03370787){89}{\line(1,0){.04494382}}
\put(19.25,55.304){\vector(-1,-2){.07}}\multiput(20.25,57.304)(-.03333333,-.06666667){60}{\line(0,-1){.06666667}}
\put(21.25,55.304){\vector(-3,-2){.07}}\multiput(24.25,57.304)(-.050420168,-.033613445){119}{\line(-1,0){.050420168}}
\put(23.25,55.304){\vector(3,-2){.07}}\multiput(20.25,57.304)(.050420168,-.033613445){119}{\line(1,0){.050420168}}
\put(25.25,55.304){\vector(1,-2){.07}}\multiput(24.25,57.304)(.03333333,-.06666667){60}{\line(0,-1){.06666667}}
\put(24.25,51.804){\vector(4,3){.07}}\multiput(22.25,50.304)(.04494382,.03370787){89}{\line(1,0){.04494382}}
\put(20.25,51.804){\vector(4,-3){.07}}\multiput(18.25,53.304)(.04494382,-.03370787){89}{\line(1,0){.04494382}}
\put(34.25,55.304){\vector(-1,2){.07}}\multiput(35.25,53.304)(-.03333333,.06666667){60}{\line(0,1){.06666667}}
\put(32.25,55.304){\vector(-3,2){.07}}\multiput(35.25,53.304)(-.050420168,.033613445){119}{\line(-1,0){.050420168}}
\put(29.25,51.804){\vector(4,-3){.07}}\multiput(27.25,53.304)(.04494382,-.03370787){89}{\line(1,0){.04494382}}
\put(39.25,55.304){\vector(-3,-2){.07}}\multiput(42.25,57.304)(-.050420168,-.033613445){119}{\line(-1,0){.050420168}}
\put(37.25,55.304){\vector(-1,-2){.07}}\multiput(38.25,57.304)(-.03333333,-.06666667){60}{\line(0,-1){.06666667}}
\put(42.25,51.804){\vector(4,3){.07}}\multiput(40.25,50.304)(.04494382,.03370787){89}{\line(1,0){.04494382}}
\put(1.25,46.304){\vector(-1,-2){.07}}\multiput(2.25,48.304)(-.03333333,-.06666667){60}{\line(0,-1){.06666667}}
\put(3.25,46.304){\vector(-3,-2){.07}}\multiput(6.25,48.304)(-.050420168,-.033613445){119}{\line(-1,0){.050420168}}
\put(7.25,46.304){\vector(-1,2){.07}}\multiput(8.25,44.304)(-.03333333,.06666667){60}{\line(0,1){.06666667}}
\put(5.25,46.304){\vector(-3,2){.07}}\multiput(8.25,44.304)(-.050420168,.033613445){119}{\line(-1,0){.050420168}}
\put(10.25,46.304){\vector(-1,-2){.07}}\multiput(11.25,48.304)(-.03333333,-.06666667){60}{\line(0,-1){.06666667}}
\put(12.25,46.304){\vector(-3,-2){.07}}\multiput(15.25,48.304)(-.050420168,-.033613445){119}{\line(-1,0){.050420168}}
\put(20.25,42.804){\vector(4,-3){.07}}\multiput(18.25,44.304)(.04494382,-.03370787){89}{\line(1,0){.04494382}}
\put(24.25,42.804){\vector(4,3){.07}}\multiput(22.25,41.304)(.04494382,.03370787){89}{\line(1,0){.04494382}}
\put(28.25,46.304){\vector(-1,-2){.07}}\multiput(29.25,48.304)(-.03333333,-.06666667){60}{\line(0,-1){.06666667}}
\put(29.25,42.804){\vector(4,-3){.07}}\multiput(27.25,44.304)(.04494382,-.03370787){89}{\line(1,0){.04494382}}
\put(34.25,46.304){\vector(-1,2){.07}}\multiput(35.25,44.304)(-.03333333,.06666667){60}{\line(0,1){.06666667}}
\put(31.25,48.304){\vector(1,0){.07}}\put(29.25,48.304){\line(1,0){4}}
\put(38.25,47.304){\vector(-1,-1){.07}}\multiput(40.25,49.304)(-.033613445,-.033613445){119}{\line(0,-1){.033613445}}
\put(38.25,43.304){\vector(1,-1){.07}}\multiput(36.25,45.304)(.033613445,-.033613445){119}{\line(0,-1){.033613445}}
\put(39.75,47.304){\vector(-1,-4){.07}}\multiput(40.25,49.304)(-.0333333,-.1333333){30}{\line(0,-1){.1333333}}
\put(39.75,43.304){\vector(1,-4){.07}}\multiput(39.25,45.304)(.0333333,-.1333333){30}{\line(0,-1){.1333333}}
\put(40.75,47.304){\vector(1,-4){.07}}\multiput(40.25,49.304)(.0333333,-.1333333){30}{\line(0,-1){.1333333}}
\put(42.25,47.304){\vector(1,-1){.07}}\multiput(40.25,49.304)(.033613445,-.033613445){119}{\line(0,-1){.033613445}}
\put(40.75,43.304){\vector(1,4){.07}}\multiput(40.25,41.304)(.0333333,.1333333){30}{\line(0,1){.1333333}}
\put(42.25,43.304){\vector(1,1){.07}}\multiput(40.25,41.304)(.033613445,.033613445){119}{\line(0,1){.033613445}}
\put(11.25,37.304){\vector(1,1){.07}}\multiput(9.25,35.304)(.033613445,.033613445){119}{\line(0,1){.033613445}}
\put(12.75,37.304){\vector(1,4){.07}}\multiput(12.25,35.304)(.0333333,.1333333){30}{\line(0,1){.1333333}}
\put(13.75,37.304){\vector(-1,4){.07}}\multiput(14.25,35.304)(-.0333333,.1333333){30}{\line(0,1){.1333333}}
\put(15.25,37.304){\vector(-1,1){.07}}\multiput(17.25,35.304)(-.033613445,.033613445){119}{\line(0,1){.033613445}}
\put(11.25,33.304){\vector(1,-1){.07}}\multiput(9.25,35.304)(.033613445,-.033613445){119}{\line(1,0){.033613445}}
\put(12.75,33.304){\vector(1,-4){.07}}\multiput(12.25,35.304)(.0333333,-.1333333){30}{\line(0,-1){.1333333}}
\put(13.75,33.304){\vector(1,4){.07}}\multiput(13.25,31.304)(.0333333,.1333333){30}{\line(0,1){.1333333}}
\put(15.25,33.304){\vector(1,1){.07}}\multiput(13.25,31.304)(.033613445,.033613445){119}{\line(1,0){.033613445}}
\put(21.75,37.304){\vector(1,4){.07}}\multiput(21.25,35.304)(.0333333,.1333333){30}{\line(0,1){.1333333}}
\put(19.75,35.304){\vector(-1,0){.07}}\put(21.25,35.304){\line(-1,0){3}}
\put(21.75,33.304){\vector(1,-4){.07}}\multiput(21.25,35.304)(.0333333,-.1333333){30}{\line(0,-1){.1333333}}
\put(22.75,33.304){\vector(1,4){.07}}\multiput(22.25,31.304)(.0333333,.1333333){30}{\line(0,1){.1333333}}
\put(22.75,37.304){\vector(-1,4){.07}}\multiput(23.25,35.304)(-.0333333,.1333333){30}{\line(0,1){.1333333}}
\put(24.25,33.304){\vector(1,1){.07}}\multiput(22.25,31.304)(.033613445,.033613445){119}{\line(1,0){.033613445}}
\put(24.25,37.304){\vector(-1,1){.07}}\multiput(26.25,35.304)(-.033613445,.033613445){119}{\line(0,1){.033613445}}
\put(29.25,37.304){\vector(-1,-1){.07}}\multiput(31.25,39.304)(-.033613445,-.033613445){119}{\line(0,-1){.033613445}}
\put(30.75,37.304){\vector(-1,-4){.07}}\multiput(31.25,39.304)(-.0333333,-.1333333){30}{\line(0,-1){.1333333}}
\put(30.75,33.304){\vector(1,-4){.07}}\multiput(30.25,35.304)(.0333333,-.1333333){30}{\line(0,-1){.1333333}}
\put(29.25,33.304){\vector(1,-1){.07}}\multiput(27.25,35.304)(.033613445,-.033613445){119}{\line(1,0){.033613445}}
\put(38.25,37.304){\vector(-1,-1){.07}}\multiput(40.25,39.304)(-.033613445,-.033613445){119}{\line(0,-1){.033613445}}
\put(39.75,37.304){\vector(-1,-4){.07}}\multiput(40.25,39.304)(-.0333333,-.1333333){30}{\line(0,-1){.1333333}}
\put(40.75,37.304){\vector(-1,4){.07}}\multiput(41.25,35.304)(-.0333333,.1333333){30}{\line(0,1){.1333333}}
\put(42.75,35.304){\vector(1,0){.07}}\put(41.25,35.304){\line(1,0){3}}
\put(42.25,33.304){\vector(1,1){.07}}\multiput(40.25,31.304)(.033613445,.033613445){119}{\line(1,0){.033613445}}
\put(4.25,39.304){\line(-1,-1){4}}
\put(.25,35.304){\line(1,0){3}}
\multiput(3.25,35.304)(.0333333,.1333333){30}{\line(0,1){.1333333}}
\multiput(4.25,39.304)(.0333333,-.1333333){30}{\line(0,-1){.1333333}}
\put(5.25,35.304){\line(1,0){3}}
\put(.25,35.304){\line(1,-1){4}}
\multiput(5.25,35.304)(-.0333333,-.1333333){30}{\line(0,-1){.1333333}}
\multiput(4.25,31.304)(-.0333333,.1333333){30}{\line(0,1){.1333333}}
\put(3.25,35.304){\line(0,1){0}}
\put(4.25,39.304){\circle*{.5}}
\put(.25,35.304){\circle*{.5}}
\put(4.25,31.304){\circle*{.5}}
\put(8.25,35.304){\circle*{.5}}
\put(5.25,35.304){\circle*{.5}}
\put(3.25,35.304){\circle*{.5}}
\put(2.25,37.304){\vector(-1,-1){.07}}\multiput(4.25,39.304)(-.033613445,-.033613445){119}{\line(0,-1){.033613445}}
\put(2.25,33.304){\vector(1,-1){.07}}\multiput(.25,35.304)(.033613445,-.033613445){119}{\line(1,0){.033613445}}
\put(3.75,37.304){\vector(-1,-4){.07}}\multiput(4.25,39.304)(-.0333333,-.1333333){30}{\line(0,-1){.1333333}}
\put(3.75,33.304){\vector(1,-4){.07}}\multiput(3.25,35.304)(.0333333,-.1333333){30}{\line(0,-1){.1333333}}
\put(4.75,37.304){\vector(1,-4){.07}}\multiput(4.25,39.304)(.0333333,-.1333333){30}{\line(0,-1){.1333333}}
\put(4.75,33.304){\vector(1,4){.07}}\multiput(4.25,31.304)(.0333333,.1333333){30}{\line(0,1){.1333333}}
\put(4.25,41.304){\circle*{.5}}
\put(4.25,39.304){\line(1,-1){4}}
\put(8.25,35.304){\line(-1,-1){4}}
\put(4.25,29.304){\line(-1,-1){4}}
\put(.25,25.304){\line(1,0){3}}
\multiput(3.25,25.304)(.0333333,.1333333){30}{\line(0,1){.1333333}}
\multiput(4.25,29.304)(.0333333,-.1333333){30}{\line(0,-1){.1333333}}
\put(5.25,25.304){\line(1,0){3}}
\put(8.25,25.304){\line(-1,1){4}}
\put(4.25,21.304){\line(1,1){4}}
\multiput(5.25,25.304)(-.0333333,-.1333333){30}{\line(0,-1){.1333333}}
\put(3.25,25.304){\line(0,1){0}}
\put(4.25,29.304){\circle*{.5}}
\put(.25,25.304){\circle*{.5}}
\put(4.25,21.304){\circle*{.5}}
\put(8.25,25.304){\circle*{.5}}
\put(5.25,25.304){\circle*{.5}}
\put(3.25,25.304){\circle*{.5}}
\put(2.25,27.304){\vector(-1,-1){.07}}\multiput(4.25,29.304)(-.033613445,-.033613445){119}{\line(0,-1){.033613445}}
\put(3.75,27.304){\vector(-1,-4){.07}}\multiput(4.25,29.304)(-.0333333,-.1333333){30}{\line(0,-1){.1333333}}
\put(.25,25.304){\line(1,-1){4}}
\multiput(3.25,25.304)(.0333333,-.1333333){30}{\line(0,-1){.1333333}}
\put(4.25,21.304){\line(0,1){0}}
\put(9.25,25.304){\line(1,0){3}}
\multiput(13.25,29.304)(.0333333,-.1333333){30}{\line(0,-1){.1333333}}
\put(14.25,25.304){\line(1,0){3}}
\put(13.25,21.304){\line(1,1){4}}
\multiput(14.25,25.304)(-.0333333,-.1333333){30}{\line(0,-1){.1333333}}
\put(12.25,25.304){\line(0,1){0}}
\put(13.25,29.304){\circle*{.5}}
\put(9.25,25.304){\circle*{.5}}
\put(13.25,21.304){\circle*{.5}}
\put(17.25,25.304){\circle*{.5}}
\put(14.25,25.304){\circle*{.5}}
\put(12.25,25.304){\circle*{.5}}
\put(9.25,25.304){\line(1,-1){4}}
\multiput(12.25,25.304)(.0333333,-.1333333){30}{\line(0,-1){.1333333}}
\put(13.25,21.304){\line(0,1){0}}
\put(13.25,29.304){\line(-1,-1){4}}
\multiput(13.25,29.304)(-.0333333,-.1333333){30}{\line(0,-1){.1333333}}
\put(18.25,25.304){\line(1,0){3}}
\multiput(22.25,29.304)(.0333333,-.1333333){30}{\line(0,-1){.1333333}}
\put(23.25,25.304){\line(1,0){3}}
\put(26.25,25.304){\line(-1,1){4}}
\put(22.25,21.304){\line(1,1){4}}
\multiput(23.25,25.304)(-.0333333,-.1333333){30}{\line(0,-1){.1333333}}
\put(21.25,25.304){\line(0,1){0}}
\put(22.25,29.304){\circle*{.5}}
\put(18.25,25.304){\circle*{.5}}
\put(22.25,21.304){\circle*{.5}}
\put(26.25,25.304){\circle*{.5}}
\put(23.25,25.304){\circle*{.5}}
\put(21.25,25.304){\circle*{.5}}
\put(18.25,25.304){\line(1,-1){4}}
\multiput(21.25,25.304)(.0333333,-.1333333){30}{\line(0,-1){.1333333}}
\put(22.25,21.304){\line(0,1){0}}
\put(22.25,29.304){\line(-1,-1){4}}
\multiput(22.25,29.304)(-.0333333,-.1333333){30}{\line(0,-1){.1333333}}
\put(27.25,25.304){\line(1,0){3}}
\multiput(31.25,29.304)(.0333333,-.1333333){30}{\line(0,-1){.1333333}}
\put(32.25,25.304){\line(1,0){3}}
\put(35.25,25.304){\line(-1,1){4}}
\put(31.25,21.304){\line(1,1){4}}
\multiput(32.25,25.304)(-.0333333,-.1333333){30}{\line(0,-1){.1333333}}
\put(30.25,25.304){\line(0,1){0}}
\put(31.25,29.304){\circle*{.5}}
\put(27.25,25.304){\circle*{.5}}
\put(31.25,21.304){\circle*{.5}}
\put(35.25,25.304){\circle*{.5}}
\put(32.25,25.304){\circle*{.5}}
\put(30.25,25.304){\circle*{.5}}
\put(27.25,25.304){\line(1,-1){4}}
\multiput(30.25,25.304)(.0333333,-.1333333){30}{\line(0,-1){.1333333}}
\put(31.25,21.304){\line(0,1){0}}
\put(31.25,29.304){\line(-1,-1){4}}
\multiput(31.25,29.304)(-.0333333,-.1333333){30}{\line(0,-1){.1333333}}
\put(36.25,25.304){\line(1,0){3}}
\multiput(40.25,29.304)(.0333333,-.1333333){30}{\line(0,-1){.1333333}}
\put(41.25,25.304){\line(1,0){3}}
\put(44.25,25.304){\line(-1,1){4}}
\put(40.25,21.304){\line(1,1){4}}
\multiput(41.25,25.304)(-.0333333,-.1333333){30}{\line(0,-1){.1333333}}
\put(39.25,25.304){\line(0,1){0}}
\put(40.25,29.304){\circle*{.5}}
\put(36.25,25.304){\circle*{.5}}
\put(40.25,21.304){\circle*{.5}}
\put(44.25,25.304){\circle*{.5}}
\put(41.25,25.304){\circle*{.5}}
\put(39.25,25.304){\circle*{.5}}
\put(36.25,25.304){\line(1,-1){4}}
\multiput(39.25,25.304)(.0333333,-.1333333){30}{\line(0,-1){.1333333}}
\put(40.25,21.304){\line(0,1){0}}
\put(40.25,29.304){\line(-1,-1){4}}
\multiput(40.25,29.304)(-.0333333,-.1333333){30}{\line(0,-1){.1333333}}
\put(4.75,27.304){\vector(-1,4){.07}}\multiput(5.25,25.304)(-.0333333,.1333333){30}{\line(0,1){.1333333}}
\put(6.25,27.304){\vector(-1,1){.07}}\multiput(8.25,25.304)(-.033613445,.033613445){119}{\line(0,1){.033613445}}
\put(12.75,27.304){\vector(-1,-4){.07}}\multiput(13.25,29.304)(-.0333333,-.1333333){30}{\line(0,-1){.1333333}}
\put(11.25,23.304){\vector(1,-1){.07}}\multiput(9.25,25.304)(.033613445,-.033613445){119}{\line(0,-1){.033613445}}
\put(13.75,27.304){\vector(-1,4){.07}}\multiput(14.25,25.304)(-.0333333,.1333333){30}{\line(0,1){.1333333}}
\put(15.25,27.304){\vector(-1,1){.07}}\multiput(17.25,25.304)(-.033613445,.033613445){119}{\line(0,1){.033613445}}
\put(21.75,27.304){\vector(-1,-4){.07}}\multiput(22.25,29.304)(-.0333333,-.1333333){30}{\line(0,-1){.1333333}}
\put(20.25,23.304){\vector(1,-1){.07}}\multiput(18.25,25.304)(.033613445,-.033613445){119}{\line(0,-1){.033613445}}
\put(24.25,23.304){\vector(1,1){.07}}\multiput(22.25,21.304)(.033613445,.033613445){119}{\line(0,1){.033613445}}
\put(24.25,25.304){\vector(1,0){.07}}\put(23.25,25.304){\line(1,0){2}}
\put(22.75,27.304){\vector(-1,4){.07}}\multiput(23.25,25.304)(-.0333333,.1333333){30}{\line(0,1){.1333333}}
\put(30.75,27.304){\vector(-1,-4){.07}}\multiput(31.25,29.304)(-.0333333,-.1333333){30}{\line(0,-1){.1333333}}
\put(29.25,27.304){\vector(-1,-1){.07}}\multiput(31.25,29.304)(-.033613445,-.033613445){119}{\line(0,-1){.033613445}}
\put(31.75,23.304){\vector(1,4){.07}}\multiput(31.25,21.304)(.0333333,.1333333){30}{\line(0,1){.1333333}}
\put(33.25,23.304){\vector(1,1){.07}}\multiput(31.25,21.304)(.033613445,.033613445){119}{\line(0,1){.033613445}}
\put(38.25,23.304){\vector(1,-1){.07}}\multiput(36.25,25.304)(.033613445,-.033613445){119}{\line(0,-1){.033613445}}
\put(39.75,23.304){\vector(1,-4){.07}}\multiput(39.25,25.304)(.0333333,-.1333333){30}{\line(0,-1){.1333333}}
\put(40.75,27.304){\vector(-1,4){.07}}\multiput(41.25,25.304)(-.0333333,.1333333){30}{\line(0,1){.1333333}}
\put(42.25,27.304){\vector(-1,1){.07}}\multiput(44.25,25.304)(-.033613445,.033613445){119}{\line(0,1){.033613445}}
\put(7.55,53.304){\makebox(0,0)[cc]{\scriptsize$v_1$}}
\put(1.55,57.304){\makebox(0,0)[cc]{\scriptsize$v_2$}}
\put(.95,53.304){\makebox(0,0)[cc]{\scriptsize$v_3$}}
\put(3.45,50.304){\makebox(0,0)[cc]{\scriptsize$v_4$}}
\put(6.95,57.304){\makebox(0,0)[cc]{\scriptsize$v_5$}}
\put(4.25,49.304){\makebox(0,0)[cc]{\scriptsize$H_1, -\frac{\sqrt{5}+1}{2}$}}
\put(13.25,49.304){\makebox(0,0)[cc]{\scriptsize$H_2, -\frac{\sqrt{5}+1}{2}$}}
\put(22.25,49.304){\makebox(0,0)[cc]{\scriptsize$H_3, -\frac{\sqrt{5}+1}{2}$}}
\put(31.25,49.304){\makebox(0,0)[cc]{\scriptsize$H_4, -\frac{\sqrt{5}+1}{2}$}}
\put(40.80,49.884){\makebox(0,0)[cc]{\scriptsize$H_5, -\frac{\sqrt{5}+1}{2}$}}
\put(4.25,40.304){\makebox(0,0)[cc]{\scriptsize$H_6, -\frac{\sqrt{5}+1}{2}$}}
\put(13.25,41.304){\circle*{.5}}
\put(13.25,40.304){\makebox(0,0)[cc]{\scriptsize$H_7, -\frac{\sqrt{5}+1}{2}$}}
\put(22.25,40.304){\makebox(0,0)[cc]{\scriptsize$H_8, -\frac{\sqrt{5}+1}{2}$}}
\put(31.25,40.304){\makebox(0,0)[cc]{\scriptsize$H_9, -\frac{\sqrt{5}+1}{2}$}}
\put(40.25,41.304){\circle*{.5}}
\put(40.25,40.304){\makebox(0,0)[cc]{\scriptsize$H_{10}, -\frac{\sqrt{5}+1}{2}$}}
\put(4.25,30.304){\makebox(0,0)[cc]{\scriptsize$H_{11}, -\frac{\sqrt{5}+1}{2}$}}
\put(13.25,30.304){\makebox(0,0)[cc]{\scriptsize$H_{12}, -\frac{\sqrt{5}+1}{2}$}}
\put(22.25,30.304){\makebox(0,0)[cc]{\scriptsize$H_{13}, -\frac{\sqrt{5}+1}{2}$}}
\put(31.25,30.304){\makebox(0,0)[cc]{\scriptsize$H_{14}, -\frac{\sqrt{5}+1}{2}$}}
\put(40.25,30.304){\makebox(0,0)[cc]{\scriptsize$H_{15}, -\frac{\sqrt{5}+1}{2}$}}
\put(4.25,20.304){\makebox(0,0)[cc]{\scriptsize$H_{16}, -\frac{\sqrt{5}+1}{2}$}}
\put(13.25,20.304){\makebox(0,0)[cc]{\scriptsize$H_{17}, -\frac{\sqrt{5}+1}{2}$}}
\put(22.25,20.304){\makebox(0,0)[cc]{\scriptsize$H_{18}, -\frac{\sqrt{5}+1}{2}$}}
\put(31.25,20.304){\makebox(0,0)[cc]{\scriptsize$H_{19}, -\frac{\sqrt{5}+1}{2}$}}
\put(40.25,20.304){\makebox(0,0)[cc]{\scriptsize$H_{20}, -\frac{\sqrt{5}+1}{2}$}}
\put(41.00,49.100){\makebox(0,0)[cc]{\scriptsize$v_1$}}
\put(39.87,45.304){\makebox(0,0)[cc]{\scriptsize$v_2$}}
\put(40.95,41.304){\makebox(0,0)[cc]{\scriptsize$v_3$}}
\put(41.62,45.704){\makebox(0,0)[cc]{\scriptsize$v_4$}}
\put(36.25,44.604){\makebox(0,0)[cc]{\scriptsize$v_5$}}
\put(44.25,44.704){\makebox(0,0)[cc]{\scriptsize$v_6$}}
\put(11.364,42.858){\vector(4,-3){.07}}\multiput(9.364,44.358)(.04494382,-.03370787){89}{\line(1,0){.04494382}}
\put(1.25,13.304){\framebox(6,6)[cc]{}}
\put(1.25,19.304){\circle*{.5}}
\put(7.25,19.304){\circle*{.5}}
\put(1.25,13.304){\circle*{.5}}
\put(7.25,13.304){\circle*{.5}}
\put(10.25,13.304){\framebox(6,6)[cc]{}}
\put(10.25,19.304){\circle*{.5}}
\put(16.25,19.304){\circle*{.5}}
\put(10.25,13.304){\circle*{.5}}
\put(16.25,13.304){\circle*{.5}}
\put(19.25,19.304){\circle*{.5}}
\put(25.25,19.304){\circle*{.5}}
\put(19.25,13.304){\circle*{.5}}
\put(25.25,13.304){\circle*{.5}}
\put(28.25,13.304){\framebox(6,6)[cc]{}}
\put(28.25,19.304){\circle*{.5}}
\put(34.25,19.304){\circle*{.5}}
\put(28.25,13.304){\circle*{.5}}
\put(34.25,13.304){\circle*{.5}}
\put(37.25,13.304){\framebox(6,6)[cc]{}}
\put(37.25,19.304){\circle*{.5}}
\put(43.25,19.304){\circle*{.5}}
\put(37.25,13.304){\circle*{.5}}
\put(43.25,13.304){\circle*{.5}}
\put(1.25,5.304){\framebox(6,6)[cc]{}}
\put(1.25,11.304){\circle*{.5}}
\put(7.25,11.304){\circle*{.5}}
\put(1.25,5.304){\circle*{.5}}
\put(7.25,5.304){\circle*{.5}}
\put(10.25,5.304){\framebox(6,6)[cc]{}}
\put(10.25,11.304){\circle*{.5}}
\put(16.25,11.304){\circle*{.5}}
\put(10.25,5.304){\circle*{.5}}
\put(16.25,5.304){\circle*{.5}}
\put(19.25,5.304){\framebox(6,6)[cc]{}}
\put(19.25,11.304){\circle*{.5}}
\put(25.25,11.304){\circle*{.5}}
\put(19.25,5.304){\circle*{.5}}
\put(25.25,5.304){\circle*{.5}}
\put(28.25,5.304){\framebox(6,6)[cc]{}}
\put(28.25,11.304){\circle*{.5}}
\put(34.25,11.304){\circle*{.5}}
\put(28.25,5.304){\circle*{.5}}
\put(34.25,5.304){\circle*{.5}}
\put(37.25,5.304){\framebox(6,6)[cc]{}}
\put(37.25,11.304){\circle*{.5}}
\put(43.25,11.304){\circle*{.5}}
\put(37.25,5.304){\circle*{.5}}
\put(43.25,5.304){\circle*{.5}}
\put(13.25,19.304){\vector(1,0){.07}}\put(10.25,19.304){\line(1,0){6}}
\put(16.25,16.304){\vector(0,1){.07}}\put(16.25,13.304){\line(0,1){6}}
\put(22.25,19.304){\vector(-1,0){.07}}\put(25.25,19.304){\line(-1,0){6}}
\put(25.25,16.304){\vector(0,-1){.07}}\put(25.25,19.304){\line(0,-1){6}}
\put(31.25,19.304){\vector(-1,0){.07}}\put(34.25,19.304){\line(-1,0){6}}
\put(34.25,16.304){\vector(0,-1){.07}}\put(34.25,19.304){\line(0,-1){6}}
\put(28.25,16.304){\vector(0,1){.07}}\put(28.25,13.304){\line(0,1){6}}
\put(31.25,13.304){\vector(1,0){.07}}\put(28.25,13.304){\line(1,0){6}}
\put(40.25,19.304){\vector(-1,0){.07}}\put(43.25,19.304){\line(-1,0){6}}
\put(43.25,16.304){\vector(0,-1){.07}}\put(43.25,19.304){\line(0,-1){6}}
\put(37.25,16.304){\vector(0,-1){.07}}\put(37.25,19.304){\line(0,-1){6}}
\put(40.25,13.304){\vector(-1,0){.07}}\put(43.25,13.304){\line(-1,0){6}}
\put(4.25,11.304){\vector(1,0){.07}}\put(1.25,11.304){\line(1,0){6}}
\put(1.25,19.304){\line(1,-1){6}}
\put(10.25,19.304){\line(1,-1){6}}
\put(19.25,19.304){\line(1,-1){6}}
\put(28.25,19.304){\line(1,-1){6}}
\put(19.25,19.304){\line(0,-1){6}}
\put(19.25,13.304){\line(1,0){6}}
\put(37.25,19.304){\line(1,-1){6}}
\put(1.25,11.304){\line(1,-1){6}}
\put(10.25,11.304){\line(1,-1){6}}
\put(4.25,8.304){\vector(1,-1){.07}}\multiput(1.25,11.304)(.033707865,-.033707865){178}{\line(0,-1){.033707865}}
\put(1.25,8.304){\vector(0,-1){.07}}\put(1.25,11.304){\line(0,-1){6}}
\put(13.25,11.304){\vector(-1,0){.07}}\put(16.25,11.304){\line(-1,0){6}}
\put(13.25,8.304){\vector(-1,1){.07}}\multiput(16.25,5.304)(-.033707865,.033707865){178}{\line(0,1){.033707865}}
\put(10.25,8.304){\vector(0,1){.07}}\put(10.25,5.304){\line(0,1){6}}
\put(19.25,8.304){\vector(0,-1){.07}}\put(19.25,11.304){\line(0,-1){6}}
\put(22.25,5.304){\vector(1,0){.07}}\put(19.25,5.304){\line(1,0){6}}
\put(31.25,11.304){\vector(1,0){.07}}\put(28.25,11.304){\line(1,0){6}}
\put(31.25,5.304){\vector(-1,0){.07}}\put(34.25,5.304){\line(-1,0){6}}
\put(40.25,5.304){\vector(1,0){.07}}\put(37.25,5.304){\line(1,0){6}}
\put(43.25,8.304){\vector(0,1){.07}}\put(43.25,5.304){\line(0,1){6}}
\put(40.25,11.304){\vector(-1,0){.07}}\put(43.25,11.304){\line(-1,0){6}}
\put(37.25,8.304){\vector(0,1){.07}}\put(37.25,5.304){\line(0,1){6}}
\put(4.25,12.304){\makebox(0,0)[cc]{\scriptsize$H_{21}, -\frac{\sqrt{5}+1}{2}$}}
\put(13.25,12.304){\makebox(0,0)[cc]{\scriptsize$H_{22}, -\frac{\sqrt{5}+1}{2}$}}
\put(22.25,12.304){\makebox(0,0)[cc]{\scriptsize$H_{23}, -\frac{\sqrt{5}+1}{2}$}}
\put(31.25,12.304){\makebox(0,0)[cc]{\scriptsize$H_{24}, -\frac{\sqrt{5}+1}{2}$}}
\put(40.25,12.304){\makebox(0,0)[cc]{\scriptsize$H_{25}, -\frac{\sqrt{5}+1}{2}$}}
\put(3.25,4.304){\makebox(0,0)[cc]{\scriptsize$H_{26}, -\frac{\sqrt{5}+1}{2}$}}
\put(12.25,4.304){\makebox(0,0)[cc]{\scriptsize$H_{27}, -\frac{\sqrt{5}+1}{2}$}}
\put(22.25,4.304){\makebox(0,0)[cc]{\scriptsize$C_4^1, -\sqrt{2}$}}
\put(30.25,4.304){\makebox(0,0)[cc]{\scriptsize$C_4^2, -\sqrt{2}$}}
\put(40.25,4.304){\makebox(0,0)[cc]{\scriptsize$C_4^3, -\sqrt{2}$}}
\thicklines
\put(10.5,25.304){\vector(1,0){.07}}\put(9,25.304){\line(1,0){3}}
\put(19.5,25.304){\vector(1,0){.07}}\put(18,25.304){\line(1,0){3}}
\put(6.5,35.304){\vector(-1,0){.07}}\put(8,35.304){\line(-1,0){3}}
\end{picture}
\vspace{-2.5cm}
\end{center}
\end{figure}

\end{document}